\documentclass[12pt,leqno]{amsart}
\usepackage{amsmath,epsfig,graphicx,color}
\usepackage{mathrsfs,amssymb, amsfonts, latexsym, mathtools}
\usepackage{enumerate}
\usepackage{colortbl,multirow}
\usepackage{enumitem}
\definecolor{darkblue}{rgb}{.2, 0.2,.8}
\definecolor{carageen}{rgb}{0,0.5,0.3}
\definecolor{darkred}{rgb}{.8, .1,.1}

\newtheorem{lemma}{Lemma}[section]

\newtheorem{theorem}[lemma]{Theorem}

\newtheorem{proposition}[lemma]{Proposition}
\newtheorem{definition}[lemma]{Definition}
\newtheorem{corollary}[lemma]{Corollary}
\newtheorem{example}[lemma]{Example}
\newtheorem{exercise}[lemma]{Exercise}
\newtheorem{remark}[lemma]{Remark}
\newtheorem{fig}[lemma]{Figure}
\newtheorem{tab}[lemma]{Table}

\newcommand{\bth}{\begin{theorem}}
\newcommand{\ethe}{\end{theorem}}

\newcommand{\bre}{\begin{remark}\em }
\newcommand{\ere}{\end{remark}}

\newcommand{\ble}{\begin{lemma}}
\newcommand{\ele}{\end{lemma}}

\newcommand{\bde}{\begin{definition}}
\newcommand{\ede}{\end{definition}}
\newcommand{\bco}{\begin{corollary}}
\newcommand{\eco}{\end{corollary}}

\newcommand{\bpr}{\begin{proposition}}
\newcommand{\epr}{\end{proposition}}

\newcommand{\bexer}{\begin{exercise}}
\newcommand{\eexer}{\end{exercise}}

\newcommand{\bexam}{\begin{example}}
\newcommand{\eexam}{\end{example}}

\newcommand{\bfi}{\begin{fig}}
\newcommand{\efi}{\end{fig}}

\newcommand{\btab}{\begin{tab}}
\newcommand{\etab}{\end{tab}}

\newcommand{\rI}{{\rm I}}

\newcommand{\beao}{\begin{eqnarray*}}
\newcommand{\eeao}{\end{eqnarray*}\noindent}

\newcommand{\beam}{\begin{eqnarray}}
\newcommand{\eeam}{\end{eqnarray}\noindent}

\newcommand{\beqq}{\begin{equation}}
\newcommand{\eeqq}{\end{equation}\noindent}

\newcommand{\bce}{\begin{center}}
\newcommand{\ece}{\end{center}}

\newcommand{\barr}{\begin{array}}
\newcommand{\earr}{\end{array}}

\newcommand{\vague}{\stackrel{\lower0.2ex\hbox{$\scriptscriptstyle
                    \it{v} $}}{\rightarrow}}
\newcommand{\weak}{\stackrel{\lower0.2ex\hbox{$\scriptscriptstyle
                    \it{w} $}}{\rightarrow}}
\newcommand{\what}{\stackrel{\lower0.2ex\hbox{$\scriptscriptstyle
                    \it{\hat{w}} $}}{\rightarrow}}

\newcommand{\bdis}{\begin{displaymath}}
\newcommand{\edis}{\end{displaymath}\noindent}

\newcommand{\N}{\mathbb{N}}
\newcommand{\R}{\mathbb{R}}

\newcommand{\ov}{\overline}
\newcommand{\wt}{\widetilde}
\newcommand{\wh}{\widehat}
\newcommand{\vep}{\varepsilon}

\newcommand{\cals}{{\mathcal S}}
\newcommand{\call}{{\mathcal L}}
\newcommand{\cala}{{\mathcal A}}
\newcommand{\cald}{{\mathcal D}}

\allowdisplaybreaks

\begin{document}
\today
\bibliographystyle{plain}
\title[On convolution closure properties of subexponentiality]
{On convolution closure properties of subexponentiality approaching from densities}
\thanks{Muneya Matsui's research is partly supported by the JSPS Grant-in-Aid for Scientific Research C
(23K11019). 
}
\author[M. Matsui]{Muneya Matsui}
\address{Department of Business Administration, Nanzan University, 18
Yamazato-cho, Showa-ku, Nagoya 466-8673, Japan.}
\email{mmuneya@gmail.com}
\author[T. Watanabe]{Toshiro Watanabe}
\address{Center for Mathematical Sciences, The University of Aizu, 
Ikkimachi Tsuruga, Aizu-Wakamatsu, Fukushima, 965-8580, Japan.
}
\email{markov2000t@yahoo.co.jp}

\begin{abstract}
Non-closedness of subexponentiality by the convolution operation is well-known. 
We go a step further and show that subexponentiality and non-subexponentiality 
are generally changeable by the convolution. 
We also give several conditions, by which (non-) subexponentiality is kept.   
Most results are given with densities, which are easily converted to those for distributions. 
As a by-product, we give counterexamples to several past results, which were used to derive the 
non-closedness of the convolution, and modify the original proof. 
\end{abstract}
\maketitle

\section{Introduction}
Let $F,G$ be probability distribution functions on $\R:=(-\infty,\infty)$ and denote by $F\ast G$ the convolution of $F$ and $G$. 
%\[
% F\ast G(x)=\int_{-\infty}^\infty F(x-y)G(d y).
%\] 
The tail probability of $F$ is denoted by $\ov F(x)=1-F(x)$. 
Let $f,g$ be the corresponding probability density functions on $\R$ and 
denote by $f\ast g$ the convolution of $f$ and $g$. 
%\[
% f\ast g(x)=\int_{-\infty}^\infty f(x-y)g(y) d y. % \quad \text{or}\quad f^{\ast n}(x) \quad \text{for the $n$th convolution}. 
%\]
Throughout the paper, 
for functions $\alpha,\beta:\R \to \R_+:=[0,\infty)$, $\alpha(x) \sim \beta(x)$ 
means that $\lim_{x\to\infty}\alpha(x)/\beta(x)\to 1$. 
%For non-negative functions $\alpha(x)$ and $\beta(x)$, a notation $\alpha(x)\asymp \beta(x)$ implies 
%that 
%\[
% 0<\liminf_{x\to\infty} \big(\alpha(x)/ \beta(x)\big) \le \limsup_{x\to\infty} \big( \alpha(x)/\beta(x) \big) <\infty. 
%\]

We study the following tail equivalence properties for a density $f$ and a distribution $F$. 
\begin{definition} 
$(\mathrm{i})$ A non-negative measurable function $\alpha (x)$ belongs to the class $\call_\gamma,\,\gamma\ge 0$ 
if 
\[
 \lim_{x\to\infty} \alpha (x-y)/\alpha(x) = e^{\gamma y},\quad y\in \R.
\]
If $\gamma>0$, the convergence is uniform in $y\le y'$ for each $y'>0$. 
In case $\gamma=0$, we write $\call:=\call_0$. \\
$(\mathrm{ii})$ $F$ is $\gamma$-subexponential on $\R$, denoted by $\cals_\gamma$ if $\ov F \in \call_\gamma$ 
and $\ov {F^{\ast 2}}(x) \sim 2 \wh f (\gamma) \ov F(x)$, where 
\[
\wh f (\gamma) = \int_{-\infty}^{\infty}e^{\gamma x}F(dx)<\infty.
\]
In case $\gamma=0$ we write $\cals:=\cals_0$. \\
$(\mathrm{iii})$ The density $f$ of $F$ is $\gamma$-subexponential on $\R$, denoted by $\cals_\gamma$, 
if $f\in \call_\gamma$ and $f^{\ast 2}(x) \sim 2\wh f(\gamma) f(x)$. In case $\gamma=0$ we write $\cals:=\cals_0$.
\end{definition} 
Notice that the definition of $\cals_\gamma$ includes the existence of $\gamma$-moment. 
We abuse notation $\call_\gamma$ and $\cals_\gamma$, which are used for both distributions and densities. 
Since we specify either $f$ or $F$ in each context, there would be no confusion. 

The theoretical studies of distribution classes $\call_\gamma$ and $\cals_\gamma$ have been initiated by 
Chistyakov \cite{Chistyakov:1964}, Chover at al. \cite{Chover:Ney:Wainger:1973a,Chover:Ney:Wainger:1973b}, 
Embrechts \cite{Embrechts:Goldie:1980,Embrechts:Goldie:1982} and a considerable number of 
subsequent works have been conducted, see Cline \cite{Cline:1986,Cline:1987}, Kl\"uppelberg \cite{Kluppelberg:1988,Kluppelberg:1989},
Pakes \cite{Pakes:2004}, Shimura and Watanabe \cite{Shimura:Watanabe:2005}, Watanabe \cite{Watanabe:2008}, Li and Tang \cite{Li:Tang:2010} to 
name just a few. 
Even now the investigation of these classes has been an active area of research, 
and the classes continuously supply attractive research problems, which are also important in applications, 
particularly in the heavy-tailed context (e.g. Embrechts et al. \cite{embrechts:kluppelberg:mikosch:1997}, Goldie and Kl\"uppelberg \cite{Goldie:Kluppelberg:2998} 
and Foss et al. \cite{Foss:Korshunov:Zachary:2013}). 
For recent theoretical contributions, see Xu et al. \cite{Xu:Foss:Wang:2015}, Watanabe \cite{Watanabe:2019} and Leipus and \v{S}iaulys \cite{Leipus:Siaulys:2020}.  

The convolution closure problem of $\gamma$-subexponentiality $(\gamma \ge 0)$
for distributions has been raised by Embrechts and Goldie \cite{Embrechts:Goldie:1980} and negatively solved by 
Leslie \cite{Leslie:1989} for $(\gamma=0)$ and Kl\"uppelberg and Villasenor \cite{Kluppelberg:Villasenor:1991} for $(\gamma>0)$, 
i.e. they constructed $F,G \in \cals_\gamma$ such that $F\ast G \notin \cals_\gamma$.  
In the latter case, they proved non-closedness of 
subexponential densities for $\gamma=0$ (constructed $f,g\in \cals$ such that $f\ast g\notin \cals$) and used 
the subexponential relation between distributions and densities together with 
the equivalent relation of $\gamma$-subexponentiality between $\gamma=0$ and $\gamma>0$ 
(\cite[Theorem 1]{Kluppelberg:Villasenor:1991}).  
%Hence, the convolution closure problem of densities has also been 
%negtatively solved by \cite{Kluppelberg:Villasenor:1991}. 
More recently Leipus and \v{S}iaulys \cite{Leipus:Siaulys:2020} proved that 
under $F,G \in \call_\gamma$, the inclusion of $F\ast G$, $FG$ and $pF+(1-p)G$ for all (some) $p$ into 
$\cals_\gamma$ are equivalent. They also provided an example: $F,G \notin \call_1$ (so that $F,G \notin \cals_1$) such that $F\ast G \in \cals_1$. 
%More precisely  
%they constructed two subexponential densites such that 
%their convolution is not subexponential. 

In this note, focusing on densities, we tackle the convolution closure properties of $\cals_\gamma$. 
Recall that an example: $f,g\in \cals$ such that $f\ast g\notin \cals$ was given in \cite{Kluppelberg:Villasenor:1991}. 
We go a step further and show that 
subexponentiality and non-subexponentiality are generally changeable by the convolution operation, 
i.e. we show that 
\begin{enumerate}%[label=(\roman*)]
 \item the convolution of non-subexponential densities could yield a subexponential density, \\
$f,g\notin \cals \Rightarrow f\ast g\in \cals$. 
\item the convolution of a subexponential density and a non-subexponential density could yield 
both cases, $f\in \cals,g\notin\cals \Rightarrow f\ast g\in \cals$ and $f\in \cals,g\notin\cals \Rightarrow f\ast g\notin \cals$. 
\end{enumerate}
%1: the convolution of non-subexponential densities yields a subexponential density 
%$f,g\notin \cals \Rightarrow f\ast g\in \cals$, 
%2: the convolution of a subexpontial density and a non-subexpontial density could yield 
%both cases $f\in \cals,g\notin\cals \Rightarrow \{f\ast g\in \cals\} \cup \{f\ast g\notin \cals\}$. 
This means that we could exhaust all possible patterns.
We derive these cases within the class of long-tailed densities $\call$
\footnote{Notice that in the example of Leipus and \v{S}iaulys \cite{Leipus:Siaulys:2020}, $F,G\notin \call_1$}, and under the condition 
that tails of two densities do not dominate each other. 
Moreover, we prove that under some dominating relationship between the two tails, the (non-)subexponentiality of the dominant one is 
kept by the convolution. 
We notice that most of our density results here are easily converted those of distributions 
by using, e.g. l'H\^opital's rule.

In proving the results, we clarify the relations between the closure property of 
the mixture operation and that of convolution operation, which is different 
whether the densities are on $\R$ or $\R_+$. Particularly, for the 
density of $\R$, the almost decreasing property (al.d. for short, see Definition \ref{def:ani}) of $f$ plays a major role.
Indeed, we show that al.d. property is inevitable for the closure property of mixture. 

Furthermore, as a by-product, we give counterexamples for several lemmas in 
\cite{Kluppelberg:Villasenor:1991}, which are used to derive the non-closedness 
of the convolution. We modified the lemmas, so that the main assertion in \cite{Kluppelberg:Villasenor:1991} 
could be restored.

We should remark that our results are particularly inspired by previous interesting works by \cite{Kluppelberg:Villasenor:1991}, 
\cite{Leipus:Siaulys:2020}. In constructing examples, the way of making piece-wise linear densities 
by \cite{Jiang:Wang:Cui:Chen:2019} is very helpful.

In the remainder of this section, we give necessary notation and state the construction of this paper.  
Let $f_+$ be the density of the conditional distribution $F_+$ of $F$ on $\R_+$:
$f_+(x)={\bf 1}_{\R_+}(x) f(x) / \ov F(0-),\,x \in \R$ with $\ov F(0-):=\lim_{x\uparrow 0} \ov F(x)$. 
For convenience we frequently use, e.g. $\cals^c$ or $\cald^c$ which respectively implies $\notin \cals$ or $\notin \cald$. 
For non-negative functions $\alpha(x)$ and $\beta(x)$, a notation $\alpha(x)\asymp \beta(x)$ as $x\to \pm \infty$ implies 
that 
\[
 0<\liminf_{x\to \pm \infty} \big(\alpha(x)/ \beta(x)\big) \le \limsup_{x\to \pm \infty} \big( \alpha(x)/\beta(x) \big) <\infty. 
\]
The following monotonic-type conditions play a major role in the two-sided case. 
\begin{definition}
\label{def:ani}
 We say that a function $\alpha:\R\to \R_+$ is almost decreasing $($al.d.$)$ denoted by 
$\alpha \in \cald$ 
if there exists $x_0>0$ and $K>0$ such that 
\begin{align}
\label{eq:def:ald}
 \alpha(x+y) \le K\alpha(x)\quad \text{for all}\ x>x_0,\,y>0.
\end{align} 
\end{definition}
We also study the effect of the following stronger property on the subexponentiality, cf. \cite[p.23]{Bingham:Goldie:Teugels:1989}. 
\begin{definition}
 We say that a function $\alpha:\R\to \R_+$ is asymptotic to a non-increasing function $($a.n.i.$)$ denoted by 
$\alpha \in \cala$
if  
$\alpha$ is locally bounded and positive on $[x_0,\infty)$ for some $x_0>0$, and 
\begin{align}
\label{eq:def:ani}
 \sup_{t\ge x} \alpha(t)\sim \alpha(x)\quad \text{and}\quad \inf_{x_0\le t \le x} \alpha(t)\sim \alpha(x). 
\end{align}
\end{definition}

The rest of this paper is organized as follows. 
In Section \ref{sec:theory:closure:properites}, 
relations between the closure property of convolution and that of mixture 
are investigated. We also derive conditions that the (non-)subexponentiality is 
kept by the convolution. Moreover, the modification of \cite[Lemma 2]{Kluppelberg:Villasenor:1991} 
is provided. We separately give results for the positive-half case (Section \ref{sec:theory:closure:properites:positive})
and those for the two-sided case (Section \ref{sec:theory:closure:properites:two-sided}), since they are rather different. 
All proofs are summarized in Section \ref{proofs}. In Section \ref{sec:examples}
we present several examples which show the changeability of subexponentiality and non-subexponentiality 
by convolution and mixture operations. Counterexamples to Lemmas $1$ and $2$ in \cite{Kluppelberg:Villasenor:1991} are also given.

\section{The closure properties on subexponential density}
\label{sec:theory:closure:properites}
The closure properties of $\cals$ are theoretically studied. 
We separate the results for densities on $\R_+$ from those for densities on $\R$, which 
are rather different. 
At the end of this section, we show the equivalent relation between $\cals$ and $\cals_\gamma$ with $\gamma>0$ for the density, 
so that the results obtained for $\cals$ are also valid for $\cals_\gamma$. 

\subsection{Subexponential density on $\R_+$}
\label{sec:theory:closure:properites:positive}
We give two lemmas, which are also used to construct examples in Section \ref{sec:examples}. 
The first lemma relates the closure property for the mixture and that for the convolution. 
\begin{lemma}
\label{lem:conv:sum}
$(\mathrm{i})$ Let $f,g \in \call$. Then, $f\ast g\in\cals \Leftrightarrow %if and only if 
 pf+(1-p)g\in \cals$ for some $p\in(0,1)$.
If $f\ast g \in \cals$, then $f\ast g(x)\sim f(x)+g(x)$. \\
$(\mathrm{ii})$  
Let $f,g \in \cals$. Then, $f\ast g \in \cals \Leftrightarrow %if and only if $
f\ast g(x) \sim f(x)+g(x)$. 
\end{lemma}
Lemma \ref{lem:conv:sum} $(\mathrm{ii})$ is the modification of \cite[Lemma 2]{Kluppelberg:Villasenor:1991}, where only $f,g\in \call$ 
is assumed, which is not enough (see a counterexample (Ex.4) in the next section).

A motivation for the second lemma is coming from an interesting example given in \cite[Section 3]{Kluppelberg:Villasenor:1991}, 
which shows that $f,g\in \cals$ does not always imply $f\ast g\in \cals$. At there we cannot tell which has a dominant tail between 
$f$ and $g$. Then, a natural question arises: if $f\in \cals$ or $\cals^c$ and $g(x)=o(f(x))$ as $x\to\infty$, then what happens to 
$f\ast g$. The following lemma gives an answer. 

\begin{lemma}
\label{lem:comv:rplus}
 Let $f,g$ be densities on $\R_+$ such that $f\in \call$ and $g(x)=o\big(f(x)\big)$ when $x\to\infty$. Then \\
$(\mathrm{i})$ $h=f\ast g \in \cals \Leftrightarrow f\in\cals$. \\
$(\mathrm{ii})$ Assume that 
$\int_{\alpha(x)}^{x-\alpha(x)}f(x-y)g(y)dy=o\big(f(x)\big)$ 
for a function $\alpha(x)$ such that $\alpha(x)<x/2,\,\alpha(x)\to\infty$ as $x\to\infty$ and $f$ is $\alpha$-insensitive. 
%, assume that \textcolor{blue}{$I(x)=\int_{\alpha(x)}^{x-\alpha(x)}f(x-y)g(y)dy=o\big(f(x)\big)$.} 
Then $f\in \call \setminus \cals \Leftrightarrow h\in \call\setminus \cals$. 
In both cases $f(x)\sim h(x)$ holds. 
\end{lemma}
Notice that similar findings for the distribution were pointed by \cite[Proposition 1]{Embrechts:Goldie:Veraverbeke:1979}, 
\cite[Lemma 2.4]{Pakes:2004} and \cite[Proposition 1.1]{Leipus:Siaulys:2020}. 
Due to Lemma \ref{lem:comv:rplus}, the case $f,g\in \cals$ and $f\ast g\notin \cals$ occurs only if 
$f(x)\neq o(g(x))$ and $g(x)\neq o(f(x))$ hold in $\R_+$, i.e. neither $f$ nor $g$ dominates the other in the tail. 
Indeed, the example in \cite{Kluppelberg:Villasenor:1991} is included in the case, and so is 
the example that $f,g\notin \cals$ but $f\ast g\in \cals$ (Ex.3) given in Section \ref{sec:examples}.

\subsection{Subexponential density on $\R$}
\label{sec:theory:closure:properites:two-sided}
The closure properties of $\cals$ for densities on $\R$
are theoretically studied. %Different from the positive-half case, 
We need additional assumptions in the two-sided case 
to obtain the same conclusions as those for the positive-half case. 
We apply the al.d. property $\cald$ (Definition \ref{def:ani}), which would be a reasonable assumption. 
Indeed, one could see in the following lemma that 
$\cald$ is inevitable to extend the notion $\cals$ from densities on $\R_+$ to those on $\R$. 
\begin{lemma}
\label{lem:ani:R}
Let $f^+$ and $f^-$ be densities on $\R_+$ and $\R_-:=\R \setminus \R_+$, respectively and 
for $p\in(0,1)$ let $f=pf^++ (1-p)f^-$. If $f^+ \in \cals\cap \cald^c$, then
there exists $f^-$ such that $f \notin \cals$.
\end{lemma}
Take $p=F([0,\infty))$ and $f^+ =f_+$, the conditional density of $f$ on $\R_+$. 
Then, Lemma \ref{lem:ani:R} asserts that even $f_+\in \cals$, without $f\in \cald$ 
we can not rule out the possibility of $f\notin \cals$.

The following is the counterpart of Lemma \ref{lem:conv:sum}. 
\begin{lemma}
\label{lem:conv:sum:R}
 Let $f,g\in \call$. Suppose that $pf+(1-p)g\in \cals$ for some $p\in(0,1)$ and then
 $f\ast g(x)\sim f(x)+g(x)$. If additionally $pf+(1-p)g\in \cald$ for the same $p$ or 
$\limsup_{x\to-\infty} \frac{f\ast g(x)}{f(x)+g(x)}\le C$ for some $C>0$, and then $f\ast g\in \cals$. 

Conversely if $f\ast g\in \cals$ and $[$ $f\ast g\in \cald$ or $\liminf_{x\to-\infty} \frac{f\ast g(x)}{f(x)+g(x)}\ge C $ for some $C>0$
$]$, then $f\ast g(x)\sim f(x)+g(x)$, and moreover $2^{-1}(f+g)\in \cals$. 
\end{lemma}
\begin{remark}
 From Lemma \ref{lem:conv:sum:R}, we have the following fact. \\
Assume that $f,g\in \call$ and $[$ $f\ast g\in \cald$ or $f\ast g(x) \asymp f(x)+g(x)$ as $x\to-\infty$ $]$. Then 
$pf(x)+(1-p)g(x)\in \cals$ for some $p\in (0,1)$ if and only if $f\ast g\in \cals$.  
\end{remark}

The next lemma is a counterpart of Lemma\ref{lem:comv:rplus}. 
\begin{lemma}
\label{lem:comv:two-sided}
Let $f,g$ be densities on $\R$ such that $f\in \call\cap \cald$ and $g(x)=o(f(x))$. Then,\\
$(\mathrm{i})$ $f\in\cals \Leftrightarrow f\ast g \in \cals$. \\ 
$(\mathrm{ii})$ 
Assume that $\int_{\alpha(x)}^{x-\alpha(x)}f(x-y)g(y)dy=o(f(x))$ 
for a function $\alpha(x)$ such that $\alpha(x)<x/2,\,\alpha(x)\to \infty$ as $x\to\infty$ and 
$f$ is $\alpha$-insensitive. %, assume that $\int_{\alpha(x)}^{x-\alpha(x)}f(x-y)g(y)dy=o(f(x))$. 
Then $f\in \call \cap \cals^c \Leftrightarrow f\ast g \in \call\cap \cals^c$. 
\end{lemma}
\begin{remark}
Notice that in view of Lemmas \ref{lem:comv:rplus} and \ref{lem:comv:two-sided}, 
the tail domination assumption is one option to control the property $\cals$ of the mixture and the convolution, 
 which are often equivalent. If the tail domination does not hold, then we show by examples that 
all possible cases occur in the next section.  
\end{remark}
Finally, we see the equivalent relation between $\cals$ and $\cals_\gamma$ with $\gamma>0$.  
\begin{lemma}
\label{lem:exp:tilt}
 Let $f$ be a density on $\R$ and let $\gamma>0$. 
Assume that $f$ has the exponential $-\gamma$ moment, i.e. $\wh f(-\gamma)<\infty$. 
Then for the exponential tilt $f_\gamma$ of $f$, we have 
\[
 f_\gamma(x):= e^{-\gamma x} f(x)/ \wh f(-\gamma) \in \cals_\gamma \,(\text{resp.}\, \call_\gamma) 
\Leftrightarrow f \in \cals\, (\text{resp}.\, \call).  
\]
\end{lemma}
Notice that the exponential tilt is preserved by the convolution, i.e. 
\[
 f_\gamma\ast g_\gamma (x)= e^{-\gamma x} f\ast g(x)/\wh {g\ast f} (-\gamma), 
\]
so that the relation $f\ast g \in \cals \Leftrightarrow f_\gamma \ast g_\gamma \in \cals_\gamma$ is also preserved. 
This means that our theories in Section \ref{sec:theory:closure:properites} and examples in Section \ref{sec:examples} for $\cals$ are also 
applied to the study of $\cals_\gamma$. 

\section{Examples of convolution of (non-)subexponential densities}
\label{sec:examples}
In this section we investigate closedness property of the convolution together with that of the mixture
by constructing suggestive examples, where the lemmas in the previous section are intensively used. 
Counterexamples to Lemma 1 of \cite{Kluppelberg:Villasenor:1991} (Ex.5) and Lemma 2 of \cite{Kluppelberg:Villasenor:1991} 
(Ex.6) are also given.

Recall that an example of [$f,g\in \cals\, \Rightarrow\, f\ast g \in \call\cap \cals^c$] 
was found by \cite{Kluppelberg:Villasenor:1991}. Here we step further to 
the complete answer of the convolution property. We   
construct examples of $f,g,f\ast g$ such that the other three cases hold, i.e. 
\begin{align*}
f,g\in \call \cap \cals^c\, &\Rightarrow\, f\ast g \in \cals,\hspace{2cm} \text{Example 1 (one-sided)}\\ 
f\in \call \cap \cals^c ,g\in \cals \, &\Rightarrow\, f\ast g \in \call  \cap \cals^c, \hspace{0.95cm}\text{ Example 2 (one-sided)} \\
f\in \call \cap \cals^c ,g\in \cals\, &\Rightarrow\, f\ast g \in \cals, \hspace{1.9cm} \text{ Example 3 (two-sided)}
\end{align*}
where all examples are 
constructed with the condition that the tails of $f$ and $g$ do not dominate the other. 
%Notice that if we assume that $g(x)=o(f(x))$ or $f(x)=o(g(x))$
%and then through Lemmas \ref{lem:comv:rplus} and \ref{lem:comv:two-sided}, we can 
%plainly make the last two examples. However, 
In summary, the classes $\cals$ and $\cals^c$ (within $\call$) 
are interchangeable by the convolution operation when the tail domination does not occur. 
Notice that from Lemmas \ref{lem:comv:rplus} and \ref{lem:comv:two-sided}, %(convolution property), 
if a tail of $f$ or $g$ dominates the other, then (non-)subexponential property of 
the dominant one is likely to be succeeded.

%In addition, we give an example of $(\mathrm{ii})$ in 
%Lemma \ref{lem:comv:rplus}. 
%property. We give a condition that $f\in \cals \Leftrightarrow h=f\ast g $ and also 
%study  a condition  
%Namely, construct $f,g,f\ast g$ such that  $g(x)\in \cals, g(x)= o(f(x))$,  
%$f \in \call \cap \cals^c $ and $f\ast g \in \call\cap \cals^c$. 
Another aim of this section is to examine the effect of 
al.d. $(\cald)$ and a.n.i. $(\cala)$ properties on $\call$ or $\cals$.  
It is well-known that a density with a regularly varying tails satisfies 
$f\in \cals \cap \cala\, (\subset \cals\cap \cald)$, namely it is automatically subexponential 
and a.n.i. Then it remained to be seen that what happens to 
the property $\cals$ or $\call$ if $\cala$ or $\cald$ is violated. 
In \cite{Jiang:Wang:Cui:Chen:2019} a density of $\cals\cap\cald^c$ on $\R_+$ together with 
that of $(\call\setminus \cals) \cap \cald^c$ was constructed. Inspired by \cite{Jiang:Wang:Cui:Chen:2019}, we 
contract easier and newly developed examples.   
On one hand, in Ex.2, both $(\call\setminus \cals)\cap \cald^c$ (by $f$) and $\cals \cap \cald^c$ (by $g$) 
are recovered, on the other hand in Ex.6 we give an example of $(\call \setminus \cals) \cap \cala \,(\subset (\call \setminus \cals) \cap \cald)$. 
In conclusion, both al.d. and a.n.i. properties may not be decisive factors for $\call$ or $\cals$ when 
the density is on the positive-half.

The situation is different in the two-sided density case. 
Firstly due to Lemma 2, even the density $f^+$ on $\R_+$ is subexponential, if $f^+\in \cald^c$, then
we can always choose the density $f^-$ on $\R_-$ %$\R\setminus \R_+$ 
such that the mixture $f=pf^{+}+(1-p)f^-$ is not subexponential.
Moreover, if we control the subexponentiality of the convolution or the mixture  
from the right-side tail, then 
the condition $\cald$ is crucial (see Lemmas \ref{lem:conv:sum:R} and \ref{lem:comv:two-sided}). Indeed, to construct 
 Example 3, we exploit Lemmas \ref{lem:conv:sum:R} with the al.d. property $\cald$. 
However, we remark that influence of $\cald$ (or $\cala$) on the property $\cals$ in the two sided case should be 
further studied. Particularly, the effect on the convolution closure property is not completely solved. 
This is an open question.

Finally, we remark that most of our density examples could be converted to distribution examples 
by using l'H\^opital's rule. We leave this to the readers who are interested.

In these examples $c$ denotes arbitrary positive constant of which values are not of interest. If we mention the function $\alpha(x)$, 
we always assume that $\alpha(x)<x/2$ and $\alpha(x)\to\infty$ as $x\to\infty$. 
\\
%The opposit assertion ($f\in\cals$ and $g(x)=o(f(x))$ imply $f\in \cals$) plainly follows. ()\\

\noindent
{\bf Example 1}: 
$f,g\in \call \cap \cals^c\, \Rightarrow\, f\ast g \in \cals$. \vspace{2mm}\\
Let $f,g$ be densities on $\R_+$. We are starting with $h=2^{-1}(f+g)$ and show $h\in \cals$. 
Then by Lemma  \ref{lem:conv:sum} $(\mathrm{i})$ we could have $f\ast g \in \cals$, where 
exact definitions of $f$ and $g$ are given afterwards. 
For $n\ge 4$ define sequences of points $\{a_n,b_n\}$ as 
\[
 a_n=e^n,\quad b_n=2^{-1}(e^n+e^{n+1})+e^{\sqrt{n}}
\]
and set intervals 
\[
 \rI_{n,1}= (a_n,b_n],\quad \rI_{n,2}=(b_n,a_{n+1}]\quad \text{and}\quad \rI_n=\rI_{n,1}\cup \rI_{n,2}.
\]
Define a piece-wise linear function $h$ by 
\begin{align*}
h(x)= \left\{
\begin{array}{ll}
e^{-4}/16 & x\in [0,e^{-4}]  \\
%& \\
e^{-n}n^{-2} & x\in \rI_{n,1} \\
%& \\
e^{-n}n^{-2} +\frac{h(a_{n+1})-h(b_n)}{a_{n+1}-b_n}(x-b_n) & x\in \rI_{n,2}, \\
\end{array}
\right.
\end{align*}
where $h'(x)=0$ on $\rI_{n,1}$ and on $\rI_{n,2}$,  
\begin{align}
\label{h':ex1}
 h'(x) %\mid_{x\in \rI_{n,2}}
= \frac{h(a_{n+1})-h(b_n)}{a_{n+1}-b_n} = \frac{2e^{-2n}n^{-2}(e^{-1}n^2 (n+1)^{-2}-1)}{e-1-2e^{-n+\sqrt{n}}}<0. 
\end{align}
Notice that strictly speaking, $h$ does not have derivatives on points $\{a_n,b_n\}$ and only 
left-hand and right-hand derivatives exist. This does not cause any problems, 
since the values of slopes $h'(x)$ on particular intervals are crucial. In the following examples, we 
define derivatives for piece-wise linear functions without mentioning this caution.

Clearly $\int_{\R_+}h(x)dx <\infty$ (calculate $\int_{\rI_n}h(x)dx$ and take the sum). Here we ignore the normalizing constant 
and regard $h$ as a density. 
First notice that 
\[
 \wt h'(x) := \sup_{x\in \rI_{n-1,2}\cup \rI_n}|h'(x)| =o \big(\inf_{x\in \rI_n} h(x)\big)
\]
and for any $x\in \rI_n$ and $t>0$, $h(x-t)$ satisfies 
\[
 h(x)-\wt h'(x) t \le h(x-t) \le h(x)+ \wt h'(x) t,
\]
so that $h\in \call$ is clear.

Before proving that $h\in \cals$, we bound $h(x)$ by $\ell(x)=2(1+e)x^{-1}(\log x)^{-2}$ in the tail. 
By direct calculations 
\begin{align*}
h(a_n) = e^{-n}n^{-2} &< 2(1+e) e^{-n}n^{-2} =\ell(a_n), \\
h(b_n) = e^{-n}n^{-2} &< 4e^{-n}n^{-2}\{1+2(e+1)^{-1}e^{-n+\sqrt{n}}\}^{-1} \\
&\quad \times \{1+n^{-1}\log ((e+1+2e^{-n+\sqrt{n}})/2)\}^{-2} = \ell (b_n). 
\end{align*}
Since $\ell(x)$ is convex, $h$ possibly intersect $\ell$ on $\rI_{n,2}$. 
This is impossible. Based on 
\[
 \ell'(x)=-2 (1+e)x^{-2} (\log x)^{-2}- 4 (1+e)x^{-2} (\log x)^{-3},
\]
we observe that for sufficiently large $n$,  
\[
 |\ell' (a_{n+1})|> 2 e^{-1}(1+e^{-1}) e^{-2n}n^{-2}(1+n^{-1})^{-2}
\]
is larger than $|h'(x)|$ of \eqref{h':ex1}: $|h'(x)|\sim (2/e) e^{-2n}n^{-2}$, so that $\ell (x)$ is steeper than $h(x)$ on $\rI_{n,2}$.

Finally we prove $h\in \cals$, by showing that %for $x\in \rI_n=(a_n,a_{n+1}]$,  
\begin{align}
\label{h:j1:j2}
 h^{\ast 2}(x)= \Big(
2 \int_{0}^{\alpha(x)} +\int_{\alpha(x)}^{x-\alpha(x)} 
\Big) h(x-y)h(y)dy =:2J_1(x) + J_2(x) \sim 2 h(x)
\end{align}
with the function $\alpha(x)$. 
%a function $\alpha(x)<x/2$ such that $\alpha(x)\to\infty$ as $x\to\infty$ and 
%$h$ is $\alpha$-insensitive. 
For $x\in \rI_n$, take $\alpha(x)=e^{n}n^{-\delta}=:\alpha_n,\,\delta\in(0,1)$. 
%Set $\alpha(x)= e^{-n}n^{-\delta},\quad \delta \in (0,1)$ for $x\in I_n$ and 
We start to check $J_2(x)=o(h(x))$.
Observe that for $x\in \rI_n$, $y,x-y\in [\alpha_n,x-\alpha_n]$ in the integral $J_2$ 
and recall that $h(x)$ is dominated by 
a scaled version of $\ell (x)$ for sufficiently large $x$. 
%(Recall that we ignore the constant $C$, so we also ignore that for $\ell$.) 
%Observe that for $x\in I_{n,1}$ there aer four possibilities 
%\begin{align*}
%& \mathrm{i})\quad x-\alpha(x)\in \rI_{n-1,2},x\in \rI_{n,1},\quad \mathrm{ii})\, x-\alpha(x),x\in \rI_{n,1},  \\
%& \mathrm{iii})\quad x-\alpha(x)\in \rI_{n,1},x\in \rI_{n,2},\quad \mathrm{iv})\quad x-\alpha(x),x\in \rI_{n,2},
%\end{align*}
%and the quantities $h(x-y)/h(x)$ in \eqref{h:j1:j2} has bounds 
%\begin{align*}
%\frac{h(x-y)}{h(x)} \le  \left\{
%\begin{array}{ll}
%e(1-1/n)^{-2} & \mathrm{i})  \\
%& \\
%1 & \mathrm{ii}) \\
%& \\
%e(1+1/n)^2 & \mathrm{iii}),\mathrm{iv}) \\
%\end{array}
%\right.
%\end{align*}
Then,
\begin{align*}
J_2(x) &\le c \ell (\alpha_n)^2 a_{n+1} 
%&\le c e^{-n}n^{2\delta -4}(1-n^{-1}\delta \log n)^{-4} \\
 \sim c e^{-n}n^{2\delta -4}
= o\big(\inf_{x\in\rI_n} h(x)\big). 
\end{align*}
The remaining thing is to prove that $h$ is $\alpha$-insensitive. Notice that 
since $x-\alpha_n \in \rI_{n-1,2}\cup \rI_n$, we have 
\[
 h(x)-\wh h'(x)\alpha_n \le h(x-\alpha_n) \le h(x)+\wh h'(x)\alpha_n,
\]
where 
\[
 \wh h'(x)=\sup_{x \in \rI_{n-1,2}\cup \rI_n}|h'(x)| \le ce^{-2n}n^{-2},
\]
so that 
\[
 \wh h'(x) \alpha(x) \le %ce^{-2n}n^{-2} e^{n}n^{-\delta} = 
c e^{-n} n^{-(2+\delta)}=o\big(\inf_{x\in \rI_n} h(x)\big).
\]
This implies the $\alpha$-insensitivity of $h$, i.e. $J_1(x)\sim h(x)$.  

Next we proceed to the definitions of $f$ and $g$. Since they are piece-wise linear,  
we prepare several points and values of $f$ at these points, and give line segments to 
connect the points. 
Let $a_{n,i},i=0,\ldots,6$ be points such that 
\begin{align*}
& a_n=a_{n,0}=e^n,\,a_{n,1}=a_{n,5}-e^{\sqrt{n}}(3+e),\,a_{n,2}=a_{n,5}-e^{\sqrt{n}}(2+e),\, 
a_{n,3}= a_{n,5} - e^{\sqrt{n}}(1+e),\\
& a_{n,4}= a_{n,5} - e^{\sqrt{n}},\, a_{n,5}=2^{-1}(e^n+e^{n+1})\quad \text{and}\quad a_{n,6}= a_{n,5}+e^{\sqrt{n}}
\end{align*}
and set intervals 
\[
 \rI_{n,i}=(a_{n,i-1},a_{n,i}],\,i=1,\ldots,6\quad \text{and}\quad \rI_{n,7}=(a_{n,6},a_{n+1}]\quad \text{and}\quad 
\rI_n=(a_n,a_{n+1}]. 
\]
The definitions of $f$ and $g$ are based on 
two piece-wise linear functions $p,q$ on $\rI_n$, which are symmetric w.r.t. $e^{-n}n^{-2}$ on $\cup_{i=1}^6 \rI_{n,i}$ and 
they are the same on $\rI_{n,7}$. First we define 
$p$. 
Let 
\begin{align*}
& p(a_n)=p(a_{n,1})=p(a_{n,3})=p(a_{n,4})=p(a_{n,6})=e^{-n}n^{-2} \\
& p(a_{n,2})=2 e^{-n}n^{-2} -e^{-n}n^{-3-\vep},\,\vep\in(0,1)\quad \text{and}\quad p(a_{n,5})=e^{-n}n^{-3-\vep}, 
\end{align*}
and $p(x)$ is linear between these points, i.e. $p(a_{n,1}),\,p(a_{n,2})$ and $p(a_{n,3})$ constitute a mountain and 
$p(a_{n,4}),\,p(a_{n,5})$ and $p(a_{n,6})$ constitute a valley, and $p(x)$ is flat on $\rI_{n,1}\cup \rI_{n,4}$. 
Then define $q$ be such that $2^{-1}(p(x)+q(x))=e^{-n}n^{-2},\,x\in \rI_n$ holds.

For the latter use, we calculate the integrals of $p,q$ on $\rI_n$ as 
\begin{align}
\label{ex1:integral}
\begin{split}
& \int_{\rI_n}p(x)dx = \int_{\rI_n}q(x)dx \\
&\quad  =n^{-2} \Big(
\frac{e-1}{2}+e^{-n+\sqrt{n}}
\Big)+\frac{n^{-2}}{2} \Big(
1+e^{-1}\Big(\frac{n}{1+n}\Big)^2
\Big)\Big(
\frac{e-1}{2}-e^{-n+\sqrt{n}}
\Big) \\
&\quad \le cn^{-2}.
\end{split}
\end{align}
Now we define $f$ and $g$. Let $f(x)=g(x)=e^{-4}/16$ on $[0,e^{-4}]$ and for $k\ge 2$, put 
\begin{align*}
 f(x):=\left\{
\begin{array}{ll}
p(x) & x\in \rI_{2k}  \\
q(x) & x\in \rI_{2k+1} 
\end{array}
\right.\quad \text{and} \quad 
 g(x):=\left\{
\begin{array}{ll}
q(x) & x\in \rI_{2k}  \\
p(x) & x\in \rI_{2k+1},  
\end{array}
\right.
\end{align*}
so that one could see that $h=2^{-1}(f+g)$. 

The normalizing constant for both $f$ and $g$ is that for $h$.  
%\[
% C:= 1/16 + \sum_{n=4}^\infty \int_{\rI_n}p(x)dx <\infty
%\]
For convenience, we ignore the constant again and regard $f$ and $g$ as densities. 
The maximum absolute value of slopes is that of $\rI_{n,2},\rI_{n,3},\rI_{n,5}$ or $\rI_{n,6}$ (they are the same), i.e. 
\[
 f'(x) \mid_{x\in \rI_{n,2}} =e^{-n-\sqrt{n}}n^{-2}(1-n^{-1-\vep}) =o \big(\inf_{x\in \rI_n} f(x)\big),
\]
so that for sufficiently large $x$  
\[
 \sup_{x\in \rI_n}|f'(x)| \le \sup_{x\in \rI_{n-1}} |f'(x)| =o\big(
\inf_{x\in \rI_n}f(x)
\big).
\]
Therefore, for $x\in \rI_n$ and any fixed $t>0$,
\[
 f(x)-\sup_{x\in \rI_{n-1}}|f'(x)|  t \le f(x-t) \le f(x)+ \sup_{x\in \rI_{n-1}}|f'(x)| t,
\]
which implies $f\in \call$. By construction this implies $g\in \call$ also. 

To prove $f,g\notin \cals$, it is enough to see that $f\not \in \cals$ by symmetry. We consider the quantity 
\[
 \int_{\alpha(x)}^{x-\alpha(x)} \frac{f(x-y)}{f(x)}f(y)dy
\]
%where $\alpha(x)$ is a function such that $\alpha(x)<x/2$ and $\alpha(x)\to \infty$ as $x\to\infty$. 
for the function $\alpha(x)$. 
We take $\alpha(x)=\sqrt{\log x-c'}$ with 
$c':=\log(2^{-1}(e+1))$ and then with $x=a_{n^2,5}$, we have $\alpha(x)=n$. Let $n\ge 4$ be even numbers and 
then for sufficiently large $n$ 
\[
 \int_{n}^{a_{n^2,5}-n}\frac{f(a_{n^2,5}-y)}{f(a_{n^2,5})}f(y)dy \ge \int_{e^n}^{e^{n+1}}\frac{f(a_{n^2,5}-y)}{f(a_{n^2,5})} f(y)dy, 
\]
where one notices that $a_{n^2,5}-y \in [a_{n^2,3},a_{n^2,4}]$ when $ y\in [e^n,e^{n+1}]$ and 
\[
 f(a_{n^2,5}-y)=f(a_{n^2,3}) =e^{-n^2}n^{-4},
\]
so that it is not difficult to observe that 
\[
 \frac{ f(a_{n^2,5}-y)}{ f(a_{n^2,5})} \ge cn^{2(1+\vep)}. 
\]
In view of \eqref{ex1:integral} 
\[
 \int_{e^n}^{e^{n+1}}\frac{f(a_{n^2,5}-y)}{f(a_{n^2,5})} f(y)dy \ge cn^{2\vep} \to\infty\quad (\text{as}\quad n\to\infty). 
\]
Now due to \cite[Theorem 4.7 (ii)]{Foss:Korshunov:Zachary:2013}, we see that $f, g\notin \cals$. 
However, by applying Lemma \ref{lem:conv:sum} $(\mathrm{i})$ to $h=2^{-1}(f+g)\in\cals$, we also have $f\ast g\in \cals$. 
\hfill $\Box$\\
 
\noindent
{\bf Example 2}: $f \in \call \cap \cals^c,\,g\in\cals\, \Rightarrow\, f\ast g \in \call \cap \cals^c$. \\
Let $f,g$ be densities on $\R_+$, both of which are piece-wise linear.   
We prepare several points and values of $f$ and $g$ at these points, and give line segments to 
connect the points. 
First we define sequences of points $\{a_{n,i}\},\,i=0,\ldots,3$ as 
\[
a_n=a_{n,0}= e^n,%\,a_{n,1}=\frac{5}{4} e^n,
\,a_{n,1}=3/2\cdot e^n,\,a_{n,2}= 2^{-1}(e^n+e^{n+1})-e^{\sqrt{n}},\,
a_{n,3}=2^{-1}(e^n+e^{n+1})
\]
and intervals 
\[
 \rI_{n,i}=(a_{n,i-1},a_{n,i}],\,i=1,2,3\quad \text{and}\quad \rI_{n,4}=(a_{n,3},a_{n+1}]\quad \text{and}\quad \rI_n= (a_n,a_{n+1}]. 
\]
Then $f$ takes the following values on these points, 
\[
% f(a_n)=f(a_{n,1})=f(a_{n,2})=f(a_{n,4})= e^{-n}n^{-(3+\vep)}
f(a_{n,i})=e^{-n}n^{-(3+\vep)},\,\vep\in (0,1),\, i=0,1,3, 
\quad \text{and}\quad f(a_{n,2})= e^{-n}n^{-2},%\,\vep \in (0,1)
\]
and on $[0,e^4]$, $f(x)=e^{-4}4^{-(3+\vep)}$. Moreover, $f$ connects these points with line segments. 
On the other hand $g$ takes the following values on these points 
\[
%  g(a_n)=g(a_{n,2})=g(a_{n,3})=g(a_{n,4})= e^{-n}n^{-(3+\vep)}
g(a_{n,i})=e^{-n}n^{-(3+\vep)},\, i=0,\ldots,3  
\quad \text{and}\quad 
g(x)=e^{-4}4^{-(3+\vep)},\,x\in [0,e^{4}], 
%g(a_{n,1})= e^{-n} n^{-2}\{ (e-2)-(e-3)n^{-(1+\vep)}\}
\]
and $g$ is also piece-wise linear, i.e. $g$ connects these points with line segments.
%Thus on $\cup_{i=1}^4 \rI_{n,i}$, $f$ and $g$ have a mountain, respectively. 
%The are of both mountain is the same. 

The slopes of $f$ are 
\begin{align}
\label{derivative:f:ex2}
\begin{split}
f'(x)= \left\{
\begin{array}{lll}
\frac{f(a_{n,2})-f(a_{n,1})}{a_{n,2}-a_{n,1}}= 
\frac{2e^{-2n} n^{-2}(1-n^{-(1+\vep)})}{e-2-2e^{-n+\sqrt{n}}} & >0, & x\in \rI_{n,2}  \\
%& \\
\frac{f(a_{n,3})-f(a_{n,2})}{a_{n,3}-a_{n,2}}= e^{-n-\sqrt{n}} n^{-2}(n^{-(1+\vep)}-1) & <0, & x\in \rI_{n,3} \\
%& \\
\frac{f(a_{n+1})-f(a_{n,3})}{a_{n+1}-a_{n,3}}
= 2 e^{-2 n} n^{-(3+\vep)}(e-1)^{-1}\big\{e^{-1}(n/(n+1))^{3+\vep} -1\big\} &<0, & x\in \rI_{n,4} \\
%& \\
0 & & \text{otherwise},  
\end{array}
\right.
\end{split}
\end{align}
and those of $g$ are 
\begin{align*}
g'(x)= \left\{
\begin{array}{lll}
%\frac{g(a_{n,1})-g(a_{n})}{a_{n,1}-a_{n}}= 
%4e^{-2n} n^{-2}(e-2)(1-n^{-(1+\vep)}) & >0, & x\in \rI_{n,1}  \\
%& \\
%-g'(x)\mid_{x\in\rI_{n,1}} & <0, & x\in \rI_{n,2} \\
%& \\
f'(x)\mid_{x\in \rI_{n,4}} &<0, & x\in \rI_{n,4} \\
%& \\
0 & & \text{otherwise}. 
\end{array}
\right.
\end{align*}
Notice that
\[
 \sup_{x\in \rI_{n-1}\cup \rI_n} |f'(x)|=o\big(
\inf_{x\in\rI_{n}} f(x)\big)
\quad \text{and}\quad  \sup_{x\in \rI_{n-1}\cup \rI_n} |g'(x)|=o\big(
\inf_{x\in\rI_{n}}g(x)
\big), 
\]
so that both $f\in\call$ and $g\in \call$ hold. 

Next we observe the integrability of $f,g$ and calculate 
values of integrals $\int_{\rI_n}g(x)dx$ and $\int_{\rI_n}f(x)dx$. %, which 
%are turned out to be the same for common $n$. 
By direct calculations, we have  
\begin{align}
\begin{split}
\label{area:intg:intr}
\int_{\rI_n}g(x)dx &=  4^{-1}n^{-(3+\vep)} (e-1)(3+e^{-1}(1+1/n)^{-(3+\vep)}),\\
\text{and}\quad \int_{\rI_n}f(x)dx &= \int_{\rI_n}g(x)dx + 
4^{-1}n^{-2}(1-n^{-(1+\vep)})(e-2), 
\end{split}
%\big\{(e-1)(1+3n^{-(1+\vep)})+(e-2)(1-n^{-(1+\vep)}) \big\}.
\end{align}
so that they are integrable. The normalizing constants are respectively 
$c_g=4^{-(3+\vep)}+ \sum_{n=4}^\infty \int_{\rI_n}g(x)dx<\infty$
and $c_f=4^{-(3+\vep)}+\sum_{n=4}^\infty \int_{\rI_n}f(x)dx<\infty$. 
We ignore the constants and regard $f$ and $g$ as densities.

To obtain $g\in \cals$ we bound $g(x)$ by $\ell(x)=(e+1)x^{-1}(\log x)^{-2}$ in the tail. 
By a direct calculation for sufficiently large $n$ 
\begin{align*}
% g(a_{n}) = e^{-n}n^{-(3+\vep)} &<(1+e) e^{-n}n^{-2} =\ell (a_n) \\
% g(a_{n,1}) = e^{-n}n^{-2}      &< 4/5 (1+e)e^{-n}\big(n+\log (5/4)\big)^{-2} =\ell (a_{n,1})  \\
g(a_{n}) &=g(a_{n,1}) =g(a_{n,2})=g(a_{n,3}) =  e^{-n}n^{-(3+\vep)} \\
&< 2e^{-n}  \big(n+\log((e+1)/2) \big)^{-2} = \ell (a_{n,3}). 
%g(a_{n,1}) = e^{-n}n^{-2}      &< 4/5 (1+e)e^{-n}\big(n+\log (5/4)\big)^{-2} =\ell (a_{n,1}). 
\end{align*}
Since $\ell(x)$ is non-increasing and convex, $g$ possibly intersects $\ell$ on $\rI_{n,4}$. %or $\rI_{n,5}$.
This is impossible. Based on 
\[
 \ell'(x) =-(e+1) x^{-2} (\log x)^{-2}-2(e+1)x^{-2}(\log x)^{-3},
\]
we observe that for sufficiently large $n$
%\begin{align*}
% |\ell'(a_{n,1})| &= \frac{16(e+1)}{25} e^{-2n}n^{-2} (1+n^{-1} \log(5/4))^{-2} \big(1+2(n+\log (5/4))^{-1}\big) \\
%&< 4^{-2n} n^{-2} (1-n^{-(1+\vep)}) = |g'(x)|\mid_{x\in \rI_{n,2}},
%\end{align*}
%and $g(x)$ is steeper than $\ell (x)$ on $\rI_{n,2}$. Moreover,
\begin{align*}
 |\ell'(a_{n+1})| &> (e+1)e^{-2(n+1)}(n+1)^{-2} 
%&> 2e^{-2n} n^{-(3+\vep)} (e-1)^{-1}\{1-(n/(n+1))^{3+\vep}\} = 
> |g'(x)|\mid_{x\in \rI_{n,4}},
\end{align*}
so that $\ell(x)$ is steeper than $g(x)$ on $\rI_{n,4}$. 

We return to the proof of $g\in \cals$ and show that %for $x\in \rI_{n}=(a_n,a_{n+1}]$, we consider 
\begin{align*}
 g^{\ast 2}(x)= \Big(
2 \int_{0}^{\alpha(x)} +\int_{\alpha(x)}^{x-\alpha(x)} 
\Big) g(x-y)g(y)dy =:2J_1(x) + J_2(x) \sim 2 g(x)
\end{align*}
%for $0<\alpha(x)<x/2$ such that $\alpha(x)\to\infty$ as $x\to\infty$ and $g$ is $\alpha$-insensitive. 
for the function $\alpha(x)$. 
For $x\in \rI_n$, we take $\alpha(x)=e^{-n}n^{-\delta}=:\alpha_n,\,\delta\in (0,2^{-1}(1-\vep))$. 
We start to check $J_2(x)=o(g(x))$. Notice that for $x\in \rI_n,\,y,x-y\in [\alpha_n,x-\alpha_n]$ in $J_2$ and recall that  
$g(x)$ is bounded by $\ell(x)$ for sufficiently large $x$. Therefore,  
\begin{align*}
%\label{g:j2}
 J_2(x) \le \ell^2(\alpha_n) a_{n+1} \sim ce^{-n} n^{2\delta -4}=o\big(
\inf_{x\in \rI_n} g(x)
\big).
\end{align*}
Then, it suffices to observe that $g$ is $\alpha$-insensitive, i.e. $g(x-\alpha_n)\sim g(x)$ for $x\in \rI_n$. 
Observe that 
\[
 g(x)-\wh g'(x)\alpha_n \le g(x-\alpha_n) \le g(x)+\wh g'(x)\alpha_n 
\]
where 
\[
 \wh g'(x) = \sup_{x\in \rI_{n-1}\cup \rI_n} |g'(x)| \le ce^{-2 n}n^{-(3+\vep)}. 
\]
Thus 
\begin{align*}
%\label{widehat:deriv:g}
 \wh g'(x) \alpha_n \le c e^{-n} n^{-(3+\vep+\delta)} =o\big(
\inf_{x\in \rI_n}g(x)
\big), 
\end{align*}
so that $g(x)\sim g(x-\alpha_n)$. Now we establish $g\in \cals$.

Next we see that $f\in \call\cap \cals^c$. 
Again we study the following quantity 
\[
 \int_{\alpha(x)}^{x-\alpha(x)} \frac{f(x-y)}{f(x)}f(y)dy,
\]
where we take $\alpha(x)=\sqrt{\log x-c'},\,c'=\log(2^{-1}(e+1))$ and 
then with $x=a_{n^2,3}$ we  have $\alpha(x)=n$, so that for sufficiently large $n$
\[
 \int_{n}^{\alpha_{n^2,3}-n} \frac{f(a_{n^2,3}-y)}{f(a_{n^2,3})}f(y)dy \ge \int_{e^n}^{e^{n+1}} 
\frac{f(a_{n^2,3}-y)}{f(a_{n^2,3})}f(y)dy. 
\]
Observe that $a_{n^2,3}-y\in [a_{n^2,1},a_{n^2,2}]$ when $y\in [e^{n},e^{n+1}]$ and thus 
\begin{align*}
 f(a_{n^2,3}-y) &\ge f(a_{n^2,2}) - \frac{2e^{-2n^2}n^{-4} (1-n^{-2(1+\vep)})}{e-2-2e^{-n^2+n}}(e^{n+1}-e^n) \\
& \ge f(a_{n^2,2}) -c e^{-2n^2+n}n^{-4} \\
& = f(a_{n^2,2}) -o(f(a_{n^2,3})).
\end{align*}
It is not difficult to observe that 
\begin{align*}
 \frac{f(a_{n^2,3}-y)}{f(a_{n^2,3})} \ge c \frac{ e^{-n^2}n^{-4}}{e^{-n^2}n^{-2(3+\vep)}} = c n^{2(1+\vep)}
\end{align*}
and in view of \eqref{area:intg:intr} 
\[
 \int_{e^n}^{e^{n+1}} \frac{f(a_{n^2,3}-y)}{f(a_{n^2,3})}f(y)dy \ge cn^{2\vep} \to \infty\quad \text{as}\quad x=a_{n^2,3}\to\infty.
\]
Thus $f\in\cals^c$.

Now we formally write true densities as $f_0=c_f f$ and $g_0=c_g g$ with the normalizing constants 
$c_f,c_g$ respectively. 
%noticing that the normalizing constant is the same for $f$ and $g$, $h=2^{-1}(f+g)$ satisfies 
We define $h_0$ by 
\[
 h_0= c_g/(c_f+c_g)f_0 + c_f/(c_f+c_g) g_0 = c_f c_g/(c_f+c_g) (f+g), 
\]
i.e. $h_0$ is the mixture $h_0=p f_0+(1-p)g_0$ with $p=c_g/(c_f+c_g)$. 
For convenience, we focus on the function $h=(c_f+c_g)/(c_fc_g) 2^{-1} h_0=2^{-1}(f+g)$, 
which satisfies $h(x)=e^{-4}4^{-(3+\vep)}$ for $x\in[0,e^{-4}]$, 
\begin{align*}
& h(a_n)=h(a_{n,1})=h(a_{n,3})=e^{-n}n^{-(3+\vep)},\quad 
%h(a_{n,2})=2^{-1}(g(a_{n,1})+e^{-n}n^{-2}),\\ %\,\text{and}\, h(a_{n,2})= 2^{-1}e^{-n}n^{-2} \\
 \text{and}\quad h(a_{n,2})= 2^{-1}e^{-n}n^{-2}(1+n^{-(1+\vep)}) 
\end{align*}
and it is piece-wise linear, i.e. $h$ connects $h(a_{n,i})$ with line segments. 
By the same logic as that for $f$, $h\in \call\cap \cals^c$ (so that $h_0 \in \call \cap \cals^c$) 
could be concluded. 
In view of Lemma \ref{lem:conv:sum} $(\mathrm{i})$ one notices that $f_0\ast g_0\in \cals$ implies $h_0 \in \cals$. 
%(cf. \cite[Theorem 4.8]{Foss:Korshunov:Zachary:2013}). 
Thus by taking the contraposition $f_0\ast g_0 \in \cals^c$ follows. Now by ignoring the normalizing constants which are not important, 
or 
re-define $f$ and $g$, we obtain $f\ast g\in \cals^c$.  
\hfill $\Box$\\

\noindent
{\bf Example 3}: $f \in \call \cap \cals^c,\,g\in\cals\, \Rightarrow\, f\ast g \in \call \cap \cals$. \\
%Assume that $\P(X\in \R_+)=1/2$, so that the conditional density
%$f_+$ on $\R_+$ (resp. $f_-$) coincides with $f^+$ (resp. $f^-$) of Lemma \ref{lem:ani:R}. 
We start with $f$, where we separately give $f_+$ and $f_-$, so that 
$f=c_+f_++c_-f_-$ with $c_+=\int_{[0,\infty)}f(x)dx$ and $c_-=\int_{(-\infty,0)}f(x)dx$.

First we set $f_+$.  
For $n\ge 4$ define sequences of points $\{a_n,a_{n,1},a_{n,2}\}$ as 
\[
 a_n=e^n,\quad a_{n,1}=e^n+e^n/\log n,\quad a_{n,2} = 2^{-1}(e^n+e^{n+1})
\]
and set intervals 
\[
 \rI_{n,1}=(a_n,a_{n,1}],\,\rI_{n,2}=(a_{n,1},a_{n,2}],\,\rI_{n,3}=(a_{n,2},a_{n+1}]\quad \text{and}\quad \rI_n=(a_n,a_{n+1}]. 
\]
On these points $f_+$ takes the following values 
\[
 f_+(a_n)=f_+(a_{n,2})= e^{-n}n^{-2},\, f_+(a_{n,1})= e^{-n}n^{-(2+\vep)}\quad \text{with}\quad \vep\in (0,2/3)
\]
and for $x\in[0,e^4],\,f_+(x)=e^{-4}/16$. 
On the intervals between $\{a_n,a_{n,1},a_{n,2},a_{n+1}\}$, $f_+(x)$ is linear, i.e. $f$ connects the points with line segments.  
We specify their slopes: 
\begin{align*}
\begin{split}
f_+'(x)= \left\{
\begin{array}{lll}
\frac{f_+(a_{n,1})-f_+(a_{n})}{a_{n,1}-a_{n}}= 
e^{-2n}\log n\cdot n^{-2}(n^{-\vep}-1) & <0, & x\in \rI_{n,1}  \\
\frac{f_+(a_{n,2})-f_+(a_{n,1})}{a_{n,2}-a_{n,1}}= \frac{2e^{-2n} n^{-2}(1-n^{-\vep})}{e-1-2(\log n)^{-1}} & >0, & x\in \rI_{n,2} \\
\frac{f_+(a_{n+1})-f_+(a_{n,2})}{a_{n+1}-a_{n,2}}
= \frac{2e^{-2n}n^{-2}\big(e^{-1}n^2(n+1)^{-2}-1\big)}{e-1} &<0, & x\in \rI_{n,3}.  
\end{array}
\right.
\end{split}
\end{align*}
Notice that 
\[
 \sup_{x\in \rI_{n-1,3}\cup \rI_n}|f_+'(x)|=o\big(\inf_{x\in \rI_n} f_+(x)\big) 
\]
and thus $f_+\in \call$.

Next, we check the integrability of $f_+$. 
For this we observe that $f_+$ is dominated by $u(x)=2(1+e)x^{-1} (\log x)^{-2}$ in the tail.
By a direct calculation 
\begin{align*}
f_+(a_n) &= e^{-n}n^{-2} <2(1+e)e^{-n}n^{-2} =u(a_n), \\
f_+(a_{n,1}) &< f_+(a_{n,2})=e^{-2}n^{-2} < 4 \big(1+n^{-1} \log((e+1)/2) \big)^{-2} e^{-n}n^{-2}=u(a_{n,2}). 
\end{align*}
Since $u(x)$ is non-increasing and convex $f_+$ and $u$ possibly intersect on $\rI_{n,1}$ or $\rI_{n,3}$. 
This is impossible. Based on the derivative of $u$, 
\begin{align}
\label{deriv:u}
 u'(x)=-2(1+e) x^{-2} (\log x)^{-2} -4 (1+e)x^{-2} (\log x)^{-3}, 
\end{align}
for sufficiently large $n$ we have %on $\rI_{n,1}$ 
\[
 |f_+'(x)| \mid_{x\in\rI_{n,1}} =e^{-2n}\log n \cdot n^{-2} (1-n^{-\vep})>|u'(a_n)|
\]
and the slope of $f_+(x)$ is steeper than that of $u(x)$ on $\rI_{n,1}$. In the meanwhile on $\rI_{n,3}$ 
\[
 |u'(a_{n+1})|=\frac{2}{e} \frac{1+e^{-1}}{(1+n^{-1})^{2}} e^{-2n}n^{-2}\quad\text{and}\quad |f_+'(x)|=\frac{2}{e} 
\Big(
\frac{1-e^{-1}n^2(n+1)^{-2}}{1-e^{-1}}
\Big) e^{-2n}n^{-2}, 
\]
so that $u(x)$ is steeper than $f_+(x)$ on $\rI_n$ for sufficiently large $x$. 
Now for all sufficiently large $x$, $f_+(x)<u(x)$ and thus $f_+$ is integrable. Since the 
scaling constant of $f_+$ is not essential until we consider the convolution with $g$, 
we ignore it and regard $f_+$ as a density for the moment.

Finally, we prove that $f_+\in \cals$ and to this end we show that %for $x\in \rI_n=(a_n,a_{n+1}]$  
\begin{align*}
 f_+^{\ast 2}(x)&= \int_0^x f_+(x-y)f_+(y)dy = \Big(
2\int_{0}^{\alpha(x)} +\int_{\alpha(x)}^{x-\alpha(x)}
\Big) f_+(x-y)f_+(y)dy \\
&=: 2J_1(x)+J_2(x)\sim 2f_+(x) 
\end{align*}
for the function $\alpha(x)$. 
Take $\alpha(x) =e^{-n}n^{-\delta}=:\alpha_n,\,\delta\in (\vep,1-\vep/2)$ for $x\in\rI_n$. 
We start to check $J_2(x)=o(f_+(x))$. 
Observe that for $x\in \rI_{n}$, $y,x-y\in [\alpha_n,x-\alpha_n]$ and recall that $f(x)$ is dominated by 
a scaled version of $u(x)$ for sufficiently large $x>0$. Therefore 
\begin{align}
\begin{split}
\label{evaluate:j2}
 J_2(x) &\le c u^{2}(\alpha_n) a_{n+1} \le ce^{-n} n^{2\delta-4} (1-n^{-1}\delta\log n)^{-4} \\
& \sim c e^{-n}n^{2\delta-4} =o\big(\inf_{x\in \rI_n}f_+(x)\big).
\end{split}
\end{align}
Now it suffices to observe that $f_+$ is $\alpha$-insensitive, i.e. $f_+ (x-\alpha_n)\sim f_+(x)$ for $x\in \rI_n$. 
Observe that 
\[
 f_+(x)-\wt f_+'(x)\alpha_n \le f_+(x-\alpha_n) \le f_+(x)+\wt f_+'(x)\alpha_n,
\]
where 
\[
 \wt f_+'(x) =\sup_{x\in \rI_{n-1,3}\cup \rI_n} |f'_+(x)| \le c e^{-2n} n^{-2} \log n . 
\]
Thus 
\begin{align}
\label{evaluate:j1}
 \wt f_+'(x)\alpha_n \le c e^{-n} n^{-(2+\delta)} \log n =o\big(
\inf_{x\in \rI_n} f_+(x)
\big).
\end{align}
In view of \eqref{evaluate:j2} and \eqref{evaluate:j1} we establish $f_+ \in \cals$. 
Then since $f_+\in \cals \cap \cald^c$, we choose $f_-$ such that 
$f=c_+f_++c_-f_-$ and $f\in \cals^c$, which is possible by Lemma \ref{lem:ani:R}.

Next we define $g$ on $\R_+$ so that $f_++g \in \cals \cap \cald$ holds. 
Our strategy is the same as that of Ex.1. We take the same points $\{a_n,a_{n,1},a_{n,2}\}$ and 
intervals $\rI_{n,i},\,i=1,2,3$, $\rI_n$ as those for $f_+$. On these points $g$ takes values 
\[
 g(a_n) = g(a_{n,2}) =e^{-n}n^{-2}\quad \text{and}\quad g(a_{n,1})=e^{-n}n^{-2}(2-n^{-\vep}).
\]
We specify slopes of $g$ 
\begin{align*}
\begin{split}
g'(x)= \left\{
\begin{array}{lll}
\frac{g(a_{n,1})-g(a_{n})}{a_{n,1}-a_{n}}= 
e^{-2n}\log n\cdot n^{-2}(1-n^{-\vep}) & >0, & x\in \rI_{n,1}  \\
%& \\
\frac{g(a_{n,2})-g(a_{n,1})}{a_{n,2}-a_{n,1}}= \frac{2e^{-2n} n^{-2}(n^{-\vep}-1)}{e-1-2(\log n)^{-1}} & <0, & x\in \rI_{n,2} \\
%& \\
f'(x) \mid_{x\in\rI_{n,3}} &<0, & x\in \rI_{n,3}.  
\end{array}
\right.
\end{split}
\end{align*}
Observe that 
\[
 \sup_{x\in \rI_{n-1,3}\cup \rI_n}|g'(x)|=o\big(\inf_{x\in \rI_n} g(x)\big).  
\]
Then, similarly as before, we see $g\in\ \call$. We observe that $g$ is dominated by $u(x)$ (defined above) in the tail. 
By direct calculations,  
\begin{align*}
 g(a_n) &= e^{-n}n^{-2} < 2(1+e)e^{-n}n^{-2} = u(a_n), \\
 g(a_{n,1}) &= e^{-n}n^{-2} (2-n^{-\vep}) < \frac{2(1+e)}{(1+(\log n)^{-1})(1+n^{-1} \log(1+(\log n)^{-1}))^2}e^{-n}n^{-2}
=u(a_{n,1}), \\
 g(a_{n,2}) &= e^{-n}n^{-2} < 4 (1+n^{-1} \log ((e+1)/2))^{-2} e^{-n}n^{-2} =u(a_{n,2}). 
\end{align*}
Since $u (x)$ is convex, $g$ and $u$ possibly intersect on $\rI_{n,2}$ or $\rI_{n,3}$. 
Based on \eqref{deriv:u}, we observe that for sufficiently large $n$ on $\rI_{n,2}$
\[
 |g'(x)| = \frac{2 e^{-2n}n^{-2}(1-n^{-\vep})}{e-1-2 (\log n)^{-1}} \sim \frac{2}{e-1} e^{-2n}n^{-2}
\]
and on $\rI_{n,3}$ 
\begin{align*}
 |u'(a_{n,2})| &> \frac{8}{e+1} e^{-2n}n^{-2} (1+n^{-1}\log ((e+1)/2))^{-2} \sim \frac{8}{e+1} e^{-2n}n^{-2},
\end{align*}
so that $|g'(x)| <|u'(x)|$ on $\rI_{n,2}$ for a sufficiently large $x$. 
Similar to the case of $f_+$, $|g'(x)|<|u'(x)|$ on $\rI_{n,3}$. 
Thus $u(x)$ dominates $g(x)$ for sufficiently large $x$, and $g$ is integrable. 
Again we ignore the scaling constant for $g$ for the time being.

Finally, we prove that $g\in \cals$ and again we show that for $x\in \rI_n=(a_n,a_{n+1}]$,  
\begin{align*}
 g^{\ast 2}(x) &= \Big(
2 \int_{0}^{\alpha(x)} + \int_{\alpha(x)}^{x-\alpha(x)}
\Big) g(x-y)g(y)dy \\
& =: 2J_1(x) +J_2(x) \sim g(x) 
\end{align*}
for the function $\alpha(x)$. Here we take the same $\alpha(x)=\alpha_n=e^n n^{-\delta}$ for $x\in \rI_n$ as that of $f$ but in this 
case $\delta\in (0,1)$. 
We start to check that $J_2(x)=o(g(x))$. Recall that for $x\in \rI_n$, $y,x-y \in [\alpha_n,x-\alpha_n]$ and $g(x)$ 
is dominated by a scaled version of $u(x)$ for sufficiently large $x$. Therefore 
\begin{align*}
 J_2 (x) &\le cu^2(\alpha_n) a_{n+1} \le c e^{-n}n^{2\delta-4}(1-n^{-1} \delta \log n)^{-4} \\
& \sim c e^{-n}n^{2\delta -4} =o\big(
\inf_{x\in \rI_n} g(x)
\big).
\end{align*}
Now it suffices to observe that $g$ is $\alpha_n$-insensitive, i.e. 
$g(x-\alpha_n)\sim g(x)$ for $x\in \rI_n$. Observe that 
\[
 g(x)-\wt g'(x) \alpha_n\le g(x-\alpha_n) \le g(x) +\wt g'(x) \alpha_n,
\]
where 
\[
 \wt g'(x) = \sup_{x\in \rI_{n-1,3}\cup \rI_n}|g'(x)| \le c e^{-2n} n^{-2}\log n.
\]
Noticing 
\[
 \wt g'(x)\alpha_n \le c e^{-n} n^{-(2+\delta)} \log n =o\big(
\inf_{x\in \rI_n} g(x)
\big).
\]
We see $g(x-\alpha_n)\sim g(x)$, so that $J_1(x)\sim g(x)$. Thus $g\in \cals$.

For the following argument, we denote by $c_g$ the scaling constants for $g$. 
Let $p=c_g/(c_++c_g)$ and we consider that 
\begin{align*}
 h &= p f+(1-p)c_g g \\
%   &= (f_-+c_+f_+) + (1-p) c_g g \\
   &= p c_- f_- + c_+c_g/(c_++ c_g)(f_++g). 
\end{align*}
Then $h_+(x)=c_h(f_++g),\,c_h=(c_+ + c_g)^{-1}$ satisfies
\[
 h_+(x)= 2 c_h e^{-n}n^{-2},\quad x \in [a_n,a_{n,2}],   
\]
and it is piece-wise linear. In a similar manner to $h$ of Ex.1 or $g$, we could show that 
$h_+\in \cals$ (we omit the proof) and $h\in \cald$ is immediate. 
Since $h_+$ is the conditional density of $h$ on $\R_+$, by \cite[Lemma 4.13]{Foss:Korshunov:Zachary:2013} we see that $h \in \cals$. 
Now due to Lemma \ref{lem:conv:sum:R}, $f\ast g\in \cals$ is concluded. Here we ignore the constant $c_g$ of $g$ as done before. \hfill $\Box$\\

\noindent
{\bf Example 4}:  
$f \in \call \cap \cals^c $, $g(x)\in \cals$ with $f\ast g(x)\sim f(x)+g(x)$, but $h=f\ast g \in \call\cap \cals^c$. \\
Let $f \in \call \cap \cals^c$ be as in Ex.2 and let $g(x)=\gamma x^{\gamma-1} e^{-x^\gamma},\,\gamma\in(0,1)$, i.e. 
$g$ is a Weibull density with index $\gamma \in (0,1)$, so that $g \in \cals$ (see \cite[p.86]{Foss:Korshunov:Zachary:2013}) and 
$g(x)=o(f(x))$. Our goal is to see that $h=f\ast g \in \call\cap \cals^c$ via Lemma \ref{lem:comv:rplus} $(\mathrm{ii})$. 
This is a counterexample to \cite[Lemma 2]{Kluppelberg:Villasenor:1991}.

First we see that for sufficiently large $x$, $f$ is bounded from the bottom by 
$\underbar{f}(x)=c_f x^{-1}(\log x)^{-(3+\vep)}$ with $\vep\in (0,1)$ and $c_f$ the normalized constant for $f$, i.e. 
$f(x)\ge \underbar{f}(x)$.  
As before we ignore the constant $c_f$ w.l.o.g. Since $\underbar{f}$ is decreasing,  
for sufficiently large $n$, we observe that 
\begin{align*}
& \underbar{f}(a_{n,3}) < \underbar{f}(a_{n,1}) <\underbar{f}(a_{n}) \le 
f(a_{n})= f(a_{n,1})= f(a_{n,3}), \\
& \underbar{f}(a_{n,2}) < \underbar{f}(a_n)= f(a_n) < f(a_{n,2}). 
\end{align*}
Then, noticing that $\underbar{f}$ is convex and and $f$ is a piece-wise linear function, 
we get $f(x)\ge \underbar{f}(x)$. 

Next we see that we may take $\alpha(x)=(2\log x)^{1/\gamma}$ such that 
$f$ is $\alpha$-insensitive. This is $\alpha(x)$ of Lemma \ref{lem:comv:rplus} $(\mathrm{ii})$. 
Notice that for $x \in \rI_n=(a_n,a_{n+1}]$, $(2n)^{1/\gamma}\le \alpha(x)\le (2(n+1))^{1/\gamma}$. 
For sufficiently large $n$, we may consider $8$ patterns for positions of $(x-\alpha(x),x)$, i.e. 
\begin{align*}
& (x-\alpha(x),x) \in \rI_{n,i},\quad i=1,2,3,4, \\
& x-\alpha(x)\in \rI_{n-1,4},\,x\in \rI_{n,1}\quad \text{and}\quad x-\alpha(x)\in \rI_{n,i-1},\,x \in \rI_{n,i},\,i=2,3,4, 
\end{align*}
where these intervals are defined in Ex.2. 

For the pattern of the first line, it follows that 
\[
 f(x-\alpha(x)) = f(x)-f'(x) \alpha(x)
\]
and in view of $f'(x)$: \eqref{derivative:f:ex2} in Ex. 2, we have 
\[
 \sup_{x\in \rI_{n}}|f'(x)| \alpha(x) =o\big(
\inf_{x\in \rI_{n}}f(x)
\big).
\]
Thus, we see that $f(x-\alpha(x))\sim f(x)$. 
For the second case, we have  
\[
 f(x-\alpha(x)) = f(x)+f'(x)(d_n-x)+f'(x-\alpha(x))(x-\alpha(x)-d_n),\quad d_n=\{a_n,a_{n,1},\ldots,a_{n,4}\}
\]
(notice that $f$ is piece-wise liner) and thus 
\[
 f(x)-2\sup_{x\in \rI_{n-1,4}\cup\rI_{n}}|f'(x)|\alpha(x) \le f(x-\alpha(x)) \le f(x) + 2\sup_{x\in \rI_{n-1,4}\cup\rI_{n}} |f'(x)|\alpha(x) .
\]
Again in view of $f'(x)$: \eqref{derivative:f:ex2} in Ex. 2, particularly on $\rI_{n-1,4}$, we obtain 
\[
 \sup_{x\in \rI_{n-1,4}\cup \rI_n}|f'(x)| \alpha(x) =o\big(
\inf_{x\in \rI_{n}}f(x)
\big),
\]
so that $f(x-\alpha(x))\sim f(x)$. 

Now for $\alpha(x)=(2\log x)^{1/\gamma}$ we will see the condition $(\mathrm{ii})$ of 
Lemma \ref{lem:comv:rplus}. Since $g(x)$ is decreasing 
\begin{align*}
 \int_{\alpha(x)}^{x-\alpha(x)} f(x-y)g(y)dy %&= \int_{\alpha(x)}^{x-\alpha(x)} g(x-y)f(y)dy \\
&\le g(\alpha(x)) \int_{\alpha(x)}^{x-\alpha(x)}f(y)dy \\
&\le \gamma (2\log x)^{1-1/\gamma} x^{-2} \int_{\alpha(x)}^{x-\alpha(x)} f(y)dy \\
&= o\big(\underbar{f}(x)\big)=o\big(f(x)\big), 
\end{align*}
and the condition of Lemma \ref{lem:comv:rplus} $(\mathrm{ii})$ is satisfied. 
From Lemma \ref{lem:comv:rplus} $f\ast g(x)\sim f(x)\sim f(x)+g(x)$.
One can directly prove $f\ast g(x)\sim f(x)+g(x)$ also, since $f,g \in \call$. 
\hfill $\Box$\\ 

%\begin{definition} Let $\gamma \ge 0$ and $F(dx)$ be a distribution on $\R$.
%(i)  $F(dx)$ belongs to the class $ \call(\gamma)$  if $\lim_{x \to \infty}\overline{F}(x+u)/\overline{F}(x)= e^{-\gamma u}$ for any $u \ge 0$.
%
%(ii) $F(dx)$ belongs to the class $\cals(\gamma)$  if  $F(dx) \in \call(\gamma)$ and
% $$\lim_{x \to \infty}\overline{F*F}(x)/\overline{F}(x)= 2\int_{-\infty}^{\infty}e^{\gamma x}F(dx)<\infty$$
%\end{definition}

\noindent
{\bf Example 5}: Counterexample to Lemma 1 of \cite{Kluppelberg:Villasenor:1991}.
Let $f$ and $g$ be the same as those in Ex.4. Define $F$ and $G$ on $\R_+$ as
\[
F(dx)=e^{-\gamma x}f(x)dx/ \wh f(-\gamma) %\int_{0}^{\infty}e^{-\gamma x}f(x)dx
\quad \text{and}\quad G(dx)=e^{-\gamma x}g(x)dx/ \wh g(-\gamma). %\int_{0}^{\infty}e^{-\gamma x}g(x)dx
\]
%and
%$$G(dx)=e^{-\gamma x}g(x)dx/\int_{0}^{\infty}e^{-\gamma x}g(x)dx.$$
Recall that in Ex.4, we have $f,g \in \call$, $f*g \notin \cals$ and $ f*g(x) \sim f(x)$ with $g(x) =o(f(x))$.
 By applying Theorem 2.1 of \cite{Watanabe:Yamamuro:2010},  we see that $F,G \in \call_\gamma$, $F*G \notin \cals_\gamma$, and
$$\overline{F*G}(x) \sim c_1e^{-\gamma x}f*g(x) \sim c_1 e^{-\gamma x}f(x)\sim  %\int_{0}^{\infty}e^{\gamma x}F(dx)
\wh f(\gamma)
\overline{F}(x)$$
with $\overline{G}(x)=o(\overline{F}(x))$ and $ \int_{0}^{\infty}e^{\gamma x}F(dx) < \infty$.
Thus, Example 5 is a counterexample to Lemma 1 of \cite{Kluppelberg:Villasenor:1991}.\hfill $\Box$ \\

\noindent
{\bf Example 6}: $f\in \call \cap \cals^c \cap \cala $. \\
Let $f$ be a piece-wise linear function on $\R_+$. 
Define sequences $\{a_n,b_n\},\,n\in \N$ as 
\[
 a_n=e^{n^2},\, b_n=e^{n^2}+e^{3n},\, a_{n+1}=e^{(n+1)^2}
\]
and set intervals 
\[
 \rI_{n,1}=(a_n,b_n],\quad \rI_{n,2}=(b_n,a_{n+1}]\quad \text{and}\quad \rI_n=(a_n,a_{n+1}].
\]
Let $f(x)=e^{-1}$ on $[0,e]$, $f(x)=e^{-n^2}n^{-4}$ on $(b_{n-1},a_n]$ and on $(a_n,b_n]$ 
$f(x)$ connects $f(a_n)=e^{-n^2}n^{-4}$ and $f(b_n)=e^{-(n+1)^2}(n+1)^{-4}$ with a line segment.
The slope of $f(x)$ on $(a_n,b_n]$ is 
% and on $x\in[a_n,b_n]$ 
%\[
% f(x)=e^{-n}n^{-2}
%\]
%and $f(x)$ is linear on $[b_n,a_{n+1}]$ and its slope is 
\begin{align*}
 f'(x) = \frac{f(b_n)-f(a_n)}{b_n-a_n} = 
e^{-(n+1)^2 - n}(n+1)^{-4}(e^{-2n} -e n^{-4}(n+1)^{4}) = o\big(\inf_{x\in \rI_n}f(x)\big).   
\end{align*}
%Thus for $x\in (a_n,b_n]$ 
%\[
% f(x)=f(a_n)+ f'(x)(x-a_n).
%\]
%Observe that 
%\[
% \sup_{x\in T_{ni}}|f'(x)| =o \big(\inf_{x\in \rI_{ni}}f(x)\big),\quad i=1,2. 
%\]
Then, by a similar argument as before we have $f\in \call$.

Next we specify the value of the integral of 
$f$ on $\rI_n$. Since %, i.e. 
\[
 \int_{(a_n,b_n]}f(x) dx = (f(a_n)+f(b_n)) (b_n-a_n)/2 = 2^{-1} e^{-n^2+3n} n^{-4}(1+e^{-2n-1}n^4(n+1)^{-4}) 
\]
and 
\[
 \int_{(b_n,a_{n+1}]} f(x)dx = 
f(a_{n+1}) (a_{n+1}-b_n)=
(n+1)^{-4}\big(
1-e^{-2n-1}-e^{3n -(n+1)^2}
\big),
\]
we obtain
\begin{align}
\label{integral:I_n:third}
 \int_{\rI_n} f(x)dx &=(n+1)^{-4}\big(
1-e^{-2n-1}-e^{3n -(n+1)^2}\big) \\
 &\quad +2^{-1} e^{-n^2+3n} n^{-4}\big(1+e^{-2n-1}n^4(n+1)^{-4}\big) \nonumber
\end{align}
and therefore 
\[
 \int_{\R_+}f(x)dx = 1+ \sum_{n=1}^n \int_{\rI_n}f(x)dx <\infty.
\]
From here we ignore the scaling constant of $f$ and 
regard $f$ as a density.

Next we consider the following quantity 
\[
 \int_{\alpha(x)}^{x-\alpha(x)} \frac{f(x-y)}{f(x)}f(y)dy,
\] 
for the function $\alpha(x)$. 
%where $\alpha(x)$ is a function such that $\alpha(x)<x/2$ and $\alpha(x)\to\infty$ as $x\to\infty$. 
We take $\alpha(x)=e^{2 \sqrt{\log x}}<x/2$ for sufficiently large $x>0$. 
Then for $x=b_n$ since $\alpha(b_n) = e^{2\sqrt{n^2+\log(1+e^{-n^2+3n})}}\le e^{4n}$, 
%we have $\alpha(x)=n+1$, so that 
%for sufficiently large $n+1$, so that for sufficiently large 
\begin{align*}
 \int_{\alpha(b_n)}^{b_n-\alpha(b_n)} \frac{f(b_n-y)}{f(b_n)}f(y)dy 
&\ge \int^{b_n-e^{4n}}_{e^{4n}} \frac{f(b_n-y)}{f(b_n)}f(y)dy \\
& \ge e^{2n+1}\frac{(n+1)^4}{n^4} \int_{e^{4n}}^{e^{4n+4\sqrt{n}+1}} f(y)dy, 
\end{align*}
where we notice that $b_{n-1} \le b_n -e^{4n}\le a_n$, so that $f(b_n-y)=e^{-n^2}n^{-4}$. 
%Observe that $e^{(n+1)^2}-y \in (a_{n^2},b_{n^2}]$ on $y \in [e^n,e^{n+1}]$ 
%\[
% f(e^{(n+1)^2}-y) =f(e^{n^2}) =e^{-n^2} n^{-4},
%\]
%and it is not difficult to observe that 
Now take $n=k^2$ for sufficiently large $k \in \N$ and we have from \eqref{integral:I_n:third} that 
\begin{align*}
  \int_{\alpha(b_n)}^{b_n-\alpha(b_n)} \frac{f(b_n-y)}{f(b_n)}f(y)dy  
&\ge e^{2k^2+1}\frac{(k^2+1)^4}{k^8} \int_{e^{(2k)^2}}^{e^{(2k+1)^2}} f(y)dy \\
& \ge c e^{2k^2} (2k+1)^{-4} \to \infty,
\end{align*}
as $k\to\infty$, so that $b_n\to \infty$. 
%\[
% \frac{f(e^{(n+1)^2}-y)}{f(e^{(n+1)^2})} = e^{2n+1}n^{-2}(n+1)^2 
%\]
%and in view of \eqref{integral:I_n:third} 
%\[
% \int_{e^n}^{e^{n+1}} \frac{f(x-y)}{f(x)}f(y)dy \ge ce^{2n}n^{-2}. 
%\]
Again by \cite[Theorem 4.7]{Foss:Korshunov:Zachary:2013}, $f\in \cals^c$. \\

\section{Proofs}
\label{proofs}
In the following proofs, if we mention the function $\alpha(x)$, we always assume that $\alpha(x)<x/2$ and $\alpha(x)\to\infty$ as $x\to\infty$. 
\begin{proof}[Proof of Lemma \ref{lem:conv:sum}]
Notice that $f\ast g\in \call$ (cf. \cite[Theorem 4.3]{Foss:Korshunov:Zachary:2013}), and we assume that all $f,g,f\ast g$ are $\alpha$-insensitive 
throughout the proof. \\ %we fix  
%$\alpha(x)<x/2$ such that $\alpha(x)\to \infty$ as $x\to\infty$ and all $f,g,f\ast g$ are $\alpha$-insensitive \\
$(\mathrm{i})$ First we prove the last assertion. 
Since $f,g \in \call$, 
\begin{align}
\label{liminf:fastg:fg}
 \liminf_{x\to\infty}  f\ast g(x)/(f(x)+g(x))=1
\end{align}
(\cite[Theorem 4.2]{Foss:Korshunov:Zachary:2013}), and then observe that for sufficiently large $x>0$ 
\begin{align*}
& \int_{\alpha(x)}^{x-\alpha(x)} f\ast g (x-y) f\ast g (y)dy \\
%& = \int_{\alpha(x)}^{x-\alpha(x)} \frac{f\ast g(x-y)}{(f+g)(x-y)} \frac{f\ast g(y)}{(f+g)(y)} 
%(f+g)(x-y)(f+g)(y) dy \\
&\ge (1-\vep) \int_{\alpha(x)}^{x-\alpha(x)} 
(f+g)(x-y)(f+g)(y) dy \quad \text{for some}\quad \vep \in(0,1).
\end{align*}
Since $f\ast g\in \cals$, this implies 
\begin{align}
\label{fg:smal:fastg}
 \int_{\alpha(x)}^{x-\alpha(x)} 
f(x-y) g(y) dy=o(f\ast g (x)).  
\end{align}
Now we apply \eqref{fg:smal:fastg} to the inequality 
\begin{align*}
 f\ast g(x) &\le \big(f(x)+g(x)\big) \max\Big(
\int_0^{\alpha(x)} \frac{f(x-y)}{f(x)}g(y)dy, \int_0^{\alpha(x)} \frac{g(x-y)}{g(x)}f(y)dy
\Big) \\
&\quad + 2 \int_{\alpha(x)}^{x-\alpha(x)} f(x-y) g (y)dy 
\end{align*}
and obtain  
\begin{align}
\label{limsup:fastg:f+g}
 \limsup_{x\to\infty} f\ast g(x)/(f(x)+g(x)) \le 1.
\end{align}
This together with \eqref{liminf:fastg:fg} implies the result.

Next we turn to the first assertion and show the $(\Leftarrow)$ part. 
Observe that  
\begin{align}
\label{def:pf(1-p)g}
& (pf+(1-p)g)^{\ast 2}(x) \\
&= p^2 f^{\ast 2}(x) + 2p(1-p) f\ast g(x) +(1-p)^2 g^{\ast 2}(x) \nonumber \\
&= 2 p^2 \int_0^{\alpha(x)} f(x-y) f(y)dy + 2(1-p)^2 \int_0^{\alpha(x)} g(x-y) g(y)dy \nonumber \\
&\quad + 2p(1-p)\int_0^{\alpha(x)} (f(x-y)g(y)+g(x-y)f(y))dy \nonumber \\
&\quad + \underbrace{\int_{\alpha(x)}^{x-\alpha(x)} \big\{p^2f(x-y)f(y)+(1-p)^2g(x-y)g(y)
+2p(1-p)f(x-y)g(y)
\big\}dy}_{:= I(x)} \nonumber \\
& \sim 2pf(x)+2(1-p)g(x)+I(x). \nonumber
\end{align}
Since $pf+(1-p)g\in \cals $,  
\[
 \lim_{x\to\infty} I(x)/(pf(x)+(1-p)g(x)) =0
\]
should hold, so that 
\begin{align}
\label{negligible:temrs}
 \int_{\alpha(x)}^{x-\alpha(x)} f(x-y)g(y)dy \le I(x) =o(pf(x)+(1-p)g(x))=o(f(x)+g(x)). 
\end{align}
From this we have 
\begin{align}
\label{approx:fastg:fplusg}
 f\ast g(x) &= \int_{0}^{\alpha(x)}\{ f(x-y)g(y)+g(x-y)f(y)\} dy + \int_{\alpha(x)}^{x-\alpha(x)} f(x-y)g(y)dy \\
&\sim f(x)+g(x). \nonumber 
\end{align}
Now put $p=2^{-1}$ in \eqref{def:pf(1-p)g} and apply \eqref{negligible:temrs}: $I(x)=o(f(x)+g(x))$, we obtain 
$2^{-1}(f+g) \in \cals$,
so that %\eqref{negligible:temrs} and 
\eqref{approx:fastg:fplusg} together with \cite[Theorem 4.8]{Foss:Korshunov:Zachary:2013}
yield $f\ast g \in \cals$. \\

Finally we show the $(\Rightarrow)$ part. 
Since $f\ast g(x)\sim f(x)+g(x)$, we have in \eqref{def:pf(1-p)g} that 
\begin{align*}
(pf+(1-p)g)^{\ast 2}(x) %&\sim p^2 f^{\ast 2}(x) + 2p(1-p)(f(x)+g(x)) +(1-p)^2 g^{\ast 2}(x) \\
&\sim 2p^2 \int_{0}^{\alpha(x)} f(x-y)f(y)dy + 2p(1-p)(f(x)+g(x)) \\
&\quad + 2 (1-p)^2 \int_{0}^{\alpha(x)} g(x-y)g(y)dy \\
&\quad + \int_{\alpha(x)}^{x-\alpha(x)} (p^2 f(x-y)f(y) +(1-p)^2 g(x-y)g(y))dy.
\end{align*}
Recalling the argument to obtain \eqref{fg:smal:fastg}, 
we bound the last integral by 
\[
 c \int_{\alpha(x)}^{x-\alpha(x)} (f(x-y)f(y)+g(x-y)g(y)) dy = o(pf(x)+(1-p)g(x)). 
\]
This yields 
\[
 (pf+(1-p)g)^{\ast 2}(x) \sim 2(pf(x)+(1-p)g(x)),
\]
and the proof is over. 

\noindent 
 $(\mathrm{ii})$ It suffices to show that $(\Leftarrow)$ part. 
In view of \eqref{approx:fastg:fplusg}, the condition implies 
%For the first assertion, take $\alpha(x)$ as defined in $(\mathrm{i})$ but here we additionally require that 
%$f\ast g $ is also $\alpha$-insensitive. 
%In view of \eqref{negligible:temrs}, the condition implies 
\[
\int_{\alpha(x)}^{x-\alpha(x)} f(x-y)g(y)dy =o(f(x)+g(x)).  
\]
Then noticing $f,g\in \cals$, we have 
\begin{align*}
 \int_{\alpha(x)}^{x-\alpha(x)} f\ast g(x-y) f\ast g (y)dy
&\sim \int_{\alpha(x)}^{x-\alpha(x)} (f+ g)(x-y) (f+ g) (y)dy\\
&= \int_{\alpha(x)}^{x-\alpha(x)} f(x-y)f(y)+g(x-y)g(y)dy \\
&\quad + 2 \int_{\alpha(x)}^{x-\alpha(x)} f(x-y)g(y)dy \\
&= o(f(x)+g(x))=o(f\ast g(x)).
\end{align*}
Therefore $f\ast g\in\cals$ follows from \cite[Theorem 4.7]{Foss:Korshunov:Zachary:2013}. 
\end{proof}

%A motivation of the next result is again coming from an interesting example given in \cite{Kluppelberg:Villasenor:1991}, 
%which shows that $f,g \in \cals \cap \cala$ does not always imply $h=f\ast g \in \cals$. 
%In view of the example \cite[Section 3]{Kluppelberg:Villasenor:1991}, we can not tell 
%which has a dominant tail between $f$ and $g$ of $h=f\ast g$. 
%Then a natural question is that if $f\in \cals\,\text{or}\,\cals^c$ and $g(x)= o(f(x))$ as $x\to\infty$, then  
%what happens for $h=f\ast g$. Here we do not assume $g\in \cals$. 
%Can $h\notin \cals$ occur also when the tail of one dominates the other.   
%The following lemma gives an answer. 

\begin{proof}[Proof of Lemma \ref{lem:comv:rplus}]
We use notation $h=f\ast g$. \\
$(\mathrm{i})$ $(\Leftarrow)$ Take $\alpha$ as in the condition $(\mathrm{ii})$. 
Since $f\in \cals$ and $g(x)=o\big(f(x)\big)$,  
\begin{align*}
 h(x) %&= \int_0^x f(x-y)g(y)dy \\
      &= \int_0^{\alpha(x)} f(x-y)g(y)dy + \int_0^{\alpha(x)} \frac{g(x-y)}{f(x-y)}f(x-y)f(y)dy \\
&\quad  + f(x)\int_{\alpha(x)}^{x-\alpha(x)}
\frac{g(y)}{f(y)} \frac{f(x-y) f(y)}{f(x)}dy \sim f(x),%\nonumber
\end{align*}
so that \cite[Theorem 4.8]{Foss:Korshunov:Zachary:2013} yields $h\in \cals$.  \\
$(\Rightarrow)$ Since $f\in \call$ 
\begin{align}
\label{ineq:limsup:f/h}
 \limsup_{x\to\infty}f(x)/h(x)\le 1
\end{align}
and we will consider 
\begin{align}
\label{ineq:liminf:f/h}
 \liminf_{x\to\infty} f(x)/h(x) \ge 1. 
\end{align}
Take $\alpha$ as before but we also require that $h$ 
is $\alpha$-insensitive. Observe that 
\begin{align*}
 1 &= \frac{f(x)}{h(x)} \int_{0}^{\alpha(x)} \frac{f(x-y)}{f(x)}g(y)dy + 
\frac{f(x)}{h(x)} 
\int_{0}^{\alpha(x)} \frac{g(x-y)}{f(x-y)}\frac{f(x-y)}{f(x)}f(y)dy (:=J_1(x)) \\
&\quad + \int_{\alpha(x)}^{x-\alpha(x)} \frac{g(y)}{f(y)} \frac{f(x-y)}{h(x-y)} \frac{f(y)}{h(y)} \Big( \frac{h(x-y)h(y)}{h(x)} \Big) dy 
(:=J_2(x)). 
\end{align*} 
By \eqref{ineq:limsup:f/h}, $g(x)=o(f(x))$ and $f\in\call$, $\limsup_{x\to\infty} J_1(x)=0$, while 
by \eqref{ineq:limsup:f/h}, $g(x)=o(f(x))$ and $h\in \cals$, $\limsup_{x\to\infty} J_2(x)=0$. 
Thus we have \eqref{ineq:liminf:f/h}, so that $h(x)\sim f(x)$. Again by \cite[Theorem 4.8]{Foss:Korshunov:Zachary:2013} $f\in \cals$. \\
$(\mathrm{ii})$ 
In view of the contraposition of $(\mathrm{i})$, $f\notin \cals \Leftrightarrow h\notin \cals$ follows. 
We only see $h(x)\sim f(x)$, which implies $h \in \call$. For this 
%$(\Leftarrow)$ %Since $h\notin \cals$ follows from $(\mathrm{i})$, we see 
%$h(x)\sim f(x)$, which gives $h\in\cals$. 
we prove \eqref{ineq:liminf:f/h}. Write 
\[
 h(x)=\int_0^{\alpha(x)} f(x-y)g(y)dy + \int_0^{\alpha(x)} \frac{g(x-y)}{f(x-y)}f(x-y)g(y)dy +
 \int_{\alpha(x)}^{x-\alpha(x)} f(x-y)g(y)dy.  
%I(x).
\]
Due to condition of $(\mathrm{ii})$ and $g(x)=o(f(x))$, 
we have by Fatou's lemma 
\[
 \limsup_{x\to\infty} \frac{h(x)}{f(x)} \le \limsup_{x\to\infty} \int_0^{\alpha(x)} \frac{f(x-y)}{f(x)}g(y)dy =1,
\]
so that $h(x)\sim f(x)$. % \\
%Next we derive 
%\begin{align*}
% 1&=\frac{f(x)}{h(x)} \int_0^{\alpha(x)} \frac{f(x-y)}{f(x)} g(y)dy +\frac{f(x)}{h(x)} \int_0^{\alpha(x)}
%\frac{f(x-y)}{f(x)}\frac{g(x-y)}{f(x-y)}f(y)dy \\
%&\quad+ \frac{f(x)}{h(x)} \int_{\alpha(x)}^{x-\alpha(x)} \frac{f(x-y)}{f(x)}g(y)dy .
%\end{align*}
%In view of \eqref{ineq:limsup:f/h}, $f\in\call$ and $g(x)=o(f(x))$ the second integral is negligible, while by \eqref{ineq:limsup:f/h}
% and $I(x)=o(f(x))$, the third term is also negligible. Thus, by the same logic as before $f(x)\sim h(x)$, which implies $h \in \call$. 
\end{proof}

\begin{proof}[Proof of Lemma \ref{lem:ani:R}]
Suppose that $f^+ \in \cals\cap \cald^c$. Then there are increasing divergent positive sequences 
$(a_n)$,$(b_n)$ and $(c_n),\,n\ge 1$ such that
$f^+(a_n+u)/f^+(a_n)\ge c_n$ for $b_n < u <b_n+1$. %and $n \ge 1$.  
We can choose a subsequence if necessary such that
$c: =\sum_{n=1}^{\infty} 1/\sqrt{c_n} <\infty$ and $b_{n+1} >b_n +1$ for $n \ge 1$.
Define $f^-$ as
$f^-(-u)=1/c \sqrt{ c_n}$ for $b_n <  u <b_n+1 $, and %$n \ge 1$ 
otherwise $f^-(-u)=0$.
Then, we have
$$\liminf_{n \to \infty}\frac{ f*f(a_n)}{f(a_n)}\ge
2(1-p)
\liminf_{n \to \infty}\int_{-b_n-1}^{-b_n}\frac{ f^+(a_n-u)f^-(u)}{f^+(a_n)}du=\infty,$$
so that we see that $f \notin \cals$.
\end{proof}

\begin{proof}[Proof of Lemma \ref{lem:conv:sum:R}]
For the function $\alpha(x)$, we assume w.l.o.g. that 
% Take $\alpha(x)<x/2$ such that $\alpha(x)\to\infty,\,(x\to\infty)$ and 
all $f,g,f\ast g$ are $\alpha$-insensitive. 
Then, similarly as in \eqref{def:pf(1-p)g} we observe from $\alpha$-insensitivity that 
\[
 (pf+(1-p)g)^{\ast 2}(x)\sim 2pf(x)+2(1-p)g(x)+ J(x),
\]
where with $I(x)$ of \eqref{def:pf(1-p)g} 
\begin{align*}
J(x) %&= 
% \int_{\alpha(x)}^{x-\alpha(x)} \big\{p^2f(x-y)f(y)+(1-p)^2g(x-y)g(y)
%+2p(1-p)f(x-y)g(y)
%\big\}dy \\
&= I(x) + \int_{-\infty}^{-\alpha(x)} \{ 2p^2 f(x-y)f(y) + 2(1-p)^2 g(x-y)g(y) \}dy \\
&\quad + 2p(1-p) \int_{-\infty}^{-\alpha(x)} (f(x-y)g(y)+g(x-y)f(y)) dy. 
\end{align*}
Since $pf+(1-p)g\in \cals$, 
\begin{align}
\label{neglidgeble:j}
 \lim_{x\to\infty} J(x)/(p f(x)+(1-p)g(x))=0 
\end{align}
should hold, so that 
\begin{align*}
& \int_{-\infty}^{-\alpha(x)} (f(x-y)g(y)+g(x-y)f(y)) dy + \int_{\alpha(x)}^{x-\alpha(x)} f(x-y)g(y)dy \\
& =o(pf(x)+(1-p)g(x))=o(f(x)+g(x)).
\end{align*}
Using this, we obtain $f\ast g(x)\sim f(x)+g(x)$.

Next, we show $f\ast g\in \cals$. Observe that 
\begin{align}
\label{rep:fastg}
 (f\ast g)^{\ast 2}(x) = \Big(
2\int_{-\alpha(x)}^{\alpha(x)} + 2\int_{-\infty}^{-\alpha(x)} + \int_{\alpha(x)}^{x-\alpha(x)}
\Big) f\ast g(x-y) f\ast g(y) dy 
\end{align}
and that by \eqref{neglidgeble:j} 
\begin{align}
\begin{split}
\label{rep:fastg:condi1}
 \int_{\alpha(x)}^{x-\alpha(x)} f\ast g(x-y) f\ast g(y)dy 
&\sim \int_{\alpha(x)}^{x-\alpha(x)} (f+g)(x-y)(f+g)(y)dy \\
&=o(f(x)+g(x))=o(f\ast g(x)).
\end{split}  
\end{align}
Under the first condition, $f+g\in \cald$ holds, so that 
\begin{align}
\begin{split}
\label{rep:fastg:condi2:1}
 \int_{-\infty}^{-\alpha(x)} f\ast g(x-y) f\ast g(y) dy & \sim 
 \int_{-\infty}^{-\alpha(x)} (f+g)(x-y) f\ast g(y) dy \\
 &\le c (f(x)+g(x)) \int_{-\infty}^{-\alpha(x)} f\ast g(y)dy =o(f(x)+g(x)).  
\end{split}
\end{align}
Under the second condition, it follows from \eqref{neglidgeble:j} that 
\begin{align}
\begin{split}
\label{rep:fastg:condi2:1}
 \int_{-\infty}^{-\alpha(x)} f\ast g(x-y)f\ast g(y) dy &\le c \int_{-\infty}^{-\alpha(x)} (f+g)(x-y)(f+g)(y)dy \\
& =o(f(x)+g(x)). 
\end{split}
\end{align}
Therefore, $f\ast g\in \cals$ is concluded from \eqref{rep:fastg}, \eqref{rep:fastg:condi1} and 
[\eqref{rep:fastg:condi2:1} or \eqref{rep:fastg:condi2:1}]. 
%\cite[Lemma 3.7 (v)]{matsui:2022}. 

We proceed to the converse part. Similarly as in the proof of Lemma \ref{lem:conv:sum} $(\mathrm{i})$ we could derive 
\eqref{liminf:fastg:fg} and \eqref{fg:smal:fastg}. 
First we assume $f\ast g \in \cald$ and then 
\begin{align*}
 f\ast g(x) &\le (f(x)+g(x)) \max\Big(
\int_{-\alpha(x)}^{\alpha(x)} \frac{f(x-y)}{f(x)} g(y)dy, \int_{-\alpha(x)}^{\alpha(x)} \frac{g(x-y)}{g(x)} f(y)dy
\Big) \\
&\quad + \int_{\alpha(x)}^{x-\alpha(x)}f(x-y) g(y)dy  + 2 \int_{-\infty}^{-\alpha(x)} (f(x-y)g(y)+g(x-y)f(y)) dy,
\end{align*}
where the second term is $o(f\ast g(x))$ by \eqref{fg:smal:fastg}, and the last integral is bounded by 
\begin{align*}
& 2 \int_{-\infty}^{-\alpha(x)} (f+g)(x-y)(f+g)(y)dy \\
&\quad \le c (f\ast g(x)) \int_{-\infty}^{-\alpha(x)}(f(y)+g(y))dy =o(f\ast g(x)), 
\end{align*}
where \eqref{liminf:fastg:fg} and $f\ast g \in \cald$ are used. 
Then, by exactly the same logic as in the proof of of Lemma \ref{lem:conv:sum} $(\mathrm{i})$ 
we obtain \eqref{limsup:fastg:f+g}.

Next, we prove the assertion under the second condition. Due to \eqref{liminf:fastg:fg} and 
the condition, we derive 
\begin{align*}
\int_{-\infty}^{-\alpha(x)} f\ast g(x-y) f\ast g(y)dy  \ge C(1-\vep) \int_{-\infty}^{-\alpha(x)} (f+g)(x-y)(f+g)(y)dy, 
\end{align*}
so that 
\[
 \int_{-\infty}^{-\alpha(x)} (f+g)(x-y)(f+g)(y)dy =o(f\ast g(x)). 
\]
The remaining proof is similar to the first case and we omit it. Thus, $f\ast g(x)\sim f(x)+g(x)$.

Now since 
\[
\Big( \int_{\alpha(x)}^{x-\alpha(x)} + \int_{-\infty}^{-\alpha(x)} \Big)  (f+g)(x-y)(f+g)(y)dy = o(f(x)+g(x)),
\]
we obtain $f+g\in \cals$ from the integral representation of $(f+g)^{\ast 2}$.
%\cite[Lemma 3.7 (v)]{matsui:2022}. 
\end{proof}

\begin{proof}[Proof of Lemma \ref{lem:comv:two-sided}]
$(\mathrm{i})$ $(\Rightarrow)$ immediately follows from \cite[Proposition 3.12 ({\rm i})]{matsui:2022}. \\
$(\Leftarrow)$ 
Let $h=f\ast g$. 
Since $f\in \call$ implies \eqref{ineq:limsup:f/h} we will prove \eqref{ineq:liminf:f/h}. 
%Take an insensitive function $0<\alpha(x)<x/2$ for both $h$ and $f$ and write 
For the function $\alpha(x)$, we assume w.l.o.g. that $h$ and $f$ are $\alpha$-insensitive and write
\begin{align*}
 1&=\Big(
\int_{-\infty}^{-\alpha(x)} + \int_{-\alpha(x)}^{\alpha(x)} +\int_{\alpha(x)}^{x-\alpha(x)} +
\int_{x-\alpha(x)}^{x+\alpha(x)} +\int_{x+\alpha(x)}^\infty
\Big) \frac{f(x-y) g(y)}{h(x)} dy \\
&=: I_1(x)+\cdots+I_5(x). 
\end{align*}
Since $f\in \call$, 
\[
 \liminf_{x\to\infty} I_2(x) = \liminf_{x\to\infty} \frac{f(x)}{h(x)} \int_{-\alpha(x)}^{\alpha(x)} \frac{f(x-y)}{f(x)}g(y)dy = 
\liminf_{x\to\infty} \frac{f(x)}{h(x)}.
\]
For $I_1$ we use the al.d. property of $f$ and \eqref{ineq:limsup:f/h} to obtain
\[
 I_1(x) = \frac{f(x)}{h(x)} \int_{-\infty}^{-\alpha(x)} \frac{f(x-y)}{f(x)}g(y)dy \le c\int_{-\infty}^{-\alpha(x)}g(y)dy \to 0
\]
as $x\to\infty$. 
We observe from \eqref{ineq:limsup:f/h} and $g(x)=o(f(x))$ that 
\begin{align*}
 I_3(x)=\int_{\alpha(x)}^{x-\alpha(x)} \frac{f(x-y)}{h(x-y)} \frac{f(y)}{h(y)} \frac{g(y)}{f(y)} \frac{h(x-y)h(y)}{h(x)}dy 
\le c \int_{\alpha(x)}^{x-\alpha(x)} \frac{h(x-y)h(y)}{h(x)}dy \to 0%\quad \text{as}\quad x\to\infty.
\end{align*}
as $x\to\infty$, and moreover
\begin{align*}
 I_4(x)=\frac{f(x)}{h(x)}\int_{-\alpha(x)}^{\alpha(x)} \frac{g(x-y)}{f(x-y)} \frac{f(x-y)}{f(x)} f(y)dy \le o(1) \int^{\alpha(x)}_{-\alpha(x)} f(y)dy \to 0
%\quad \text{as}\quad x\to\infty
\end{align*}
as $x\to\infty$.
Finally, since $f\in \cald$,  
\begin{align*}
 I_5(x) = \frac{f(x)}{h(x)} \int_{-\infty}^{-\alpha(x)} 
\frac{f(x-y)}{f(x)} \frac{g(x-y)}{f(x-y)} f(y)dy 
 \le c \int_{-\infty}^{-\alpha(x)} f(y)dy \to 0 %\quad \text{as}\quad x \to \infty. 
\end{align*}
as $x \to \infty$.
Thus in view of above results \eqref{ineq:liminf:f/h}, so that $h(x)\sim f(x)$ follows. 
Now by Lemma 3.11 of \cite{matsui:2022} together with $f\in \cald$, we obtain $f\in \cals$. \\

\noindent 
$(\mathrm{ii})$ Let $h=f\ast g$. By taking the contraposition of $(\mathrm{i})$, $f\in \cals^c \Leftrightarrow h \in \cals^c$ follows. 
We only observe that $h(x)\sim f(x)$ which implies $h\in \call$. For this we prove \eqref{ineq:liminf:f/h}. 
Write 
\begin{align*}
 h(x)/f(x) &= \Big(\int_{-\alpha(x)}^{\alpha(x)} + \int_{\alpha(x)}^{x-\alpha(x)}
+\int_{-\infty}^{-\alpha(x)}
\Big)
\big(
f(x-y)g(y) +g(x-y)f(y)\big)/ f(x) dy. 
%+I(x) \\&\quad + \int_{-\infty}^{-\alpha(x)}\big(f(x-y)g(y) +g(x-y)f(y)\big)dy .
\end{align*}
By condition %$I(x)\to\infty$ as $x\to\infty$
the second integral is negligible, and the last integral is bounded as 
\begin{align*}
\int_{-\infty}^{-\alpha(x)} \frac{f(x-y)}{f(x)}
\Big(
g(y)+\frac{g(x-y)}{f(x-y)}f(y)
\Big)dy \le c \int_{-\infty}^{-\alpha(x)} (g(y)+f(y))dy \to 0
\end{align*}
as $x\to\infty$. Now applying the condition $g(x)=o(f(x))$ to the first integral, we have 
\[
 \limsup_{x\to\infty} \frac{h(x)}{f(x)} \le \limsup_{x\to\infty} \int_{-\alpha(x)}^{\alpha(x)} 
\Big(
g(y)+\frac{g(x-y)}{f(x-y)}f(y)
\Big)dy =1,
\]
so that $h(x)\sim f(x)$.
\end{proof}

\begin{proof}[Proof of Lemma \ref{lem:exp:tilt}]
The proof for the equivalence of long-tailedness is immediate, 
and we only prove the equivalence of subexponentiality. \\
$(\Rightarrow)$ In view of $\wh f_\gamma(\gamma) = \wh f^{-1}(-\gamma)$, it suffices to observe 
\[
 f^{\ast 2}(x)/f(x) = f^{\ast 2}_\gamma(x)/f_\gamma(x) \cdot \wh f(-\gamma) \sim 2 \wh f_\gamma(\gamma) \wh f (-\gamma)=2. 
\]
$(\Leftarrow)$ Similarly as above, we observe 
\[
 f^{\ast 2}_\gamma(x)/f_\gamma(x) = f^{\ast 2}(x)/f(x)/ \wh f(-\gamma) \sim 2 \wh f^{-1}(-\gamma) =2 \wh f_\gamma(\gamma). 
\]
\end{proof}

%\noindent {\bf Acknowledgments}
%Muneya Matsui's research is partly supported by the JSPS Grant-in-Aid for Scientific Research C
%(19K11868). 


\begin{thebibliography}{99}
\baselineskip12pt


\bibitem{Asmussen:Foss:Korshunov:2003}
{\sc Asmussen, S., Foss, S. and Korshunov, D.}\ (2003)  
Asymptotics for sums of random variables with local subexponential behaviour. 
{\em J. Theoret. Probab.} {\bf 16}, 489--518.

\bibitem{Bingham:Goldie:Teugels:1989}
{\sc Bingham, N.H., Goldie, C.M. and Teugels, J.L.}\ (1989) 
{\em Regular variation} (No. 27). Cambridge university press.

\bibitem{Chistyakov:1964}
{\sc Chistyakov, V.P.}\ (1964) 
A theorem on sums of independent positive random variables and its applications to branching random processes. 
{\em Theory Probab. Appl.} {\bf 9}, 640--648. 

\bibitem{Chover:Ney:Wainger:1973a}
{\sc Chover, J., Ney, P. and Wainger, S.}\ (1973)
Functions of probability measures. 
{\em J. Anal. Math.} {\bf 26}, 255--302. 

\bibitem{Chover:Ney:Wainger:1973b} 
{\sc Chover. J, Ney, P. and Wainger, S.}\ (1973)
Degeneracy properties of subcritical branching processes. 
{\em Ann. Probab.} {\bf 1}, 663--673. 

\bibitem{Cline:1986}
{\sc Cline, D.B.H.}\ (1986)
Convolution tails, product tails and domains of attraction. 
{\em Probab. Theory Related Fields} {\bf 72}, 529--557.

\bibitem{Cline:1987}
{\sc Cline, D.B.H.}\ (1987)  
Convolutions of distributions with exponential and subexponential tails. 
{\em J. Aust. Math. Soc.} {\bf 43}, 347--365. 

\bibitem{Embrechts:Goldie:1980}
{\sc Embrechts, P. and Goldie, C.M.}\ (1980) 
On closure and factorization properties of subexponential and related distributions. 
{\em J. Aust. Math. Soc.} {\bf 29}, 243--256.

\bibitem{Embrechts:Goldie:1982}
{\sc Embrechts, P. and Goldie, C.M.}\ (1982) 
On convolution tails. 
{\em Stochastic Process. Appl.} {\bf 13}, 263--278.

\bibitem{Embrechts:Goldie:Veraverbeke:1979}
{\sc Embrechts, P., Goldie, C.M. and Veraverbeke, N.}\ (1979) 
Subexponentiality and infinite divisibility. 
{\em Probab. Theory Related Fields} {\bf 49}, 335--347. 
%https://doi.org/10.1007/BF00535504

\bibitem{embrechts:kluppelberg:mikosch:1997}
{\sc Embrechts, P., Kl\"uppelberg, C. and Mikosch, T.}\ (1997)
{\em Modelling extremal events for insurance and finance.}
Springer, Berlin.


\bibitem{Foss:Korshunov:Zachary:2013}
{\sc Foss, S., Korshunov, D. and Zachary, S.}\ (2013) 
{\em An introduction to heavy-tailed and subexponential distributions}.  
%(Vol. 6, pp. 0090-6778). 
New York, Springer.

\bibitem{Goldie:Kluppelberg:2998}
{\sc Goldie, C.M. and Kl\"uppelberg, C.}\ (1998). 
{\em Subexponential distributions.} 
In: A practical guide to heavy tails: statistical techniques and applications (Adler, R.J. et al. eds.)
, Birkh\"auser, 435--459.


\bibitem{Jiang:Wang:Cui:Chen:2019}
{\sc Jiang, T., Wang, Y., Cui, Z. and Chen, Y.}\ (2019)
On the almost decrease of a subexponential density. 
{\em Statist. Probab. Lett.} {\bf 153}, 71--79.  

\bibitem{Kluppelberg:1988}
{\sc Kl\"uppelberg, C.}\ (1988) 
Subexponential distributions and integrated tails. 
{\em J. Appl. Probab.} {\bf 25}, 132--141. 

\bibitem{Kluppelberg:1989}
{\sc Kl\"uppelberg, C.}\ (1989) 
Subexponential distributions and characterizations of related classes.
{\em Probab. Theory Related Fields} {\bf 82}, 259--269. 

\bibitem{Kluppelberg:Villasenor:1991}
{\sc Kl\"uppelberg, C., Villasenor, J.}\ (1991)
The full solution of the convolution closure problem for convolution-equivalent distributions.
{\em J. Math. Anal. Appl.} {\bf 160}, 79--92. 

\bibitem{Leipus:Siaulys:2020}
{\sc Leipus, R. and \v{S}iaulys, J.}\ (2020)
On a closure property of convolution equivalent class of distributions. 
{\em J. Math. Anal. Appl.} {\bf 490}, 124226,

\bibitem{Leslie:1989}
{\sc Leslie, J.R.}\ (1989). 
On the non-closure under convolution of the subexponential family. 
{\em J. Appl. Probab.} {\bf 26}, 58--66.

\bibitem{Li:Tang:2010} 
{\sc Li, J. and Tang, Q.}\ (2010)
A note on max-sum equivalence. 
{\em Statist. Probab. Lett.} {\bf 80}, 1720--1723. 

\bibitem{matsui:2022}
{\sc Matsui, M.}\ (2023)
Subexponentialiy of densities of infinitely divisible distributions. 
{\em Electron. J. Probab.} {\bf 28} 1--29

\bibitem{Pakes:2004}
{\sc Pakes, A.G.}\ (2004)  
Convolution equivalence and infinite divisibility. 
{\em J. Appl. Probab.} {\bf 41}, 407--424. 
%doi:10.1239/jap/1082999075

\bibitem{Shimura:Watanabe:2005}
{\sc Shimura, T. and Watanabe, T.}\ (2005)  
Infinite divisibility and generalized subexponentiality. 
{\em Bernoulli} {\bf 11}, 445--469.

\bibitem{Watanabe:2008}
{\sc Watanabe, T.}\ (2008) 
Convolution equivalence and distributions of random sums. 
{\em Probab. Theory Related Fields} {\bf 142}, 367--397. 
%https://doi.org/10.1007/s00440-007-0109-7


\bibitem{Watanabe:2019}
{\sc Watanabe, T.}\ (2019) 
The Wiener condition and the conjectures of Embrechts and Goldie.
{\em Ann. Probab.} {\bf 47}, 1221--1239.

\bibitem{Watanabe:Yamamuro:2010}
{\sc Watanabe, T. and Yamamuro, K.}\ (2010) 
Local subexponentiality and self-decomposability. 
{\em J. Theoret. Probab.} {\bf 23}, 1039--1067. 
%https://doi.org/10.1007/s10959-009-0240-8


\bibitem{Xu:Foss:Wang:2015}
Xu, H., Foss, S. and Wang, Y. (2015) 
Convolution and convolution-root properties of long-tailed distributions. 
{\em Extremes} {\bf 18}, 605--628.


%SSSSSSSSSSSSSSSSSSSSSSSSSSSSSSSSSSSSSSSSSSSSSSSSSSSSSSSSSSSSSSSSSSSS

\end{thebibliography}
\end{document}